\documentclass[a4,12pt]{article}%

\usepackage{graphicx}
\usepackage{graphics}

\usepackage{amsmath}
\usepackage{amsfonts}
\usepackage{amssymb}
\usepackage{enumerate}
\usepackage{xcolor}
\usepackage{tikz}
\usepackage{textcomp} 
\definecolor{Green}{rgb}{0.0, 0.5, 0.0}

\definecolor{dgreen}{rgb}{0.0, 0.5, 0.0}
\newtheorem{theorem}{Theorem}[section]

\newtheorem{corollary}[theorem]{Corollary}

\newtheorem{definition}[theorem]{Definition}
\newtheorem{example}[theorem]{Example}

\newtheorem{lemma}[theorem]{Lemma}

\newtheorem{proposition}[theorem]{Proposition}

\newtheorem{remark}[theorem]{Remark}

\newenvironment{proof}[1][Proof]{\textbf{#1.} }{\ \rule{0.5em}{0.5em}}

\newcommand{\N}{\mathbb{N}}

\newcommand{\numbercellong}[2]

\usepackage{url}
\usepackage{dirtytalk}
\begin{document}

\title{Winning Criteria for Open Games: A Game-Theoretic Approach to Prefix Codes}

\author{Dean Kraizberg\thanks{School of Mathematical Sciences, Tel Aviv University, deank@mail.tau.ac.il}}
\maketitle

\begin{abstract}
We study two-player games with alternating moves played on infinite trees. Our main focus is on the case where the trees are full (regular) and the winning set is open (with respect to the product topology on the tree). Gale and Stewart showed that in this setting one of the players always has a winning strategy, though it is not known in advance which player. We present simple necessary conditions for the first player to have a winning strategy, and establish an equivalence between winning sets that guarantee a win for the first player and maximal prefix codes. Using this equivalence, we derive a necessary algebraic condition for winning, and exhibit a family of games for which this algebraic condition is in fact equivalent to winning. We introduce the concept of coverings, and show that by covering the tree of the game with an infinite labeled tree corresponding to the free group, we can use \say{game-theoretic tools} to derive a simple trait of maximal prefix codes.
\end{abstract}


\section{Introduction}
A Gale–Stewart game is a perfect information game of two players with alternating moves. The game is defined using an alphabet set, denoted $\mathcal{A}$. The two players alternate turns, and each player is aware of all previous moves. On each turn, a player chooses a single element of $\mathcal{A}$ to play, called the action of the player on that turn. This play determines a sequence of actions $\langle a_1,a_2,... \rangle \in \mathcal{A}^{\N}$.

\medskip

Both players are assumed to have full knowledge of the winning set, denoted by
\[
W \subseteq \mathcal{A}^{\mathbb{N}}.
\]
A \emph{strategy} for a player is a choice of an action for each position on the game tree $T = \mathcal{A}^{\leq\N}$.

Player~1 is said to have a \emph{winning strategy} if there exists a strategy $s_1$ for Player~1 such that, for every strategy $s_2$ of Player~2, the infinite sequence generated by the play $x(s_1,s_2)$ belongs to $W$.
\\

Gale and Stewart~\cite{GaleStewartOpenDET} proved that whenever the winning set $W$, is open or closed in $\mathcal{A}^{\N}$, the game is always determined. 
That is, there is a player with a winning strategy. The natural question is whether there is also a criterion to determine which of the players has a winning strategy. 
Interestingly, even in the case of an open winning set, this question remained unanswered.

We focus on the special case of games with open winning sets that may be regarded as minimal in a certain sense. This focus is justified by several basic results, presented in the following section, concerning the lengths of positions in winning sets for which Player 1 has a winning strategy. We further relate our results to a recent theorem establishing a connection between winning sets of small Hausdorff dimension and the existence of a winning strategy for Player 2.

We then establish a substantive connection between maximal prefix codes and winning sets for which Player 1 has a unique winning strategy. This connection yields an equivalent condition for determining whether a given strategy of Player 1 is indeed a winning strategy in a game. Exploiting this idea, we prove several theorems (the most general of them being Theorem~\ref{corollary: using max prefix codes})—treating the binary case and the general case separately for the sake of clarity—that provide necessary and sufficient conditions for the existence of a winning strategy for Player 1 in games with simple winning sets.

Finally, drawing on structural properties of maximal prefix codes, we derive a simple algebraic criterion: if the index of a free subgroup generated (in a specified manner) by positions from an open winning set is infinite, then Player 2 has a winning strategy in the game. We also present a previously established result, which characterizes a family of games in which this condition is both necessary and sufficient.

At section~\ref{section:coverings}, we introduce the notion of coverings and present several fundamental results concerning them, as developed by Martin in his proof that Gale–Stewart games with Borel-measurable winning sets are determined~\cite{MartinCovering}. We employ this concept to cover the game tree by a labeled tree associated with free groups, and we show that this covering allows one to derive a simple result about maximal prefix codes using the tools developed in this paper.

At Section~7 we discuss some extension of the results to games in which the action set of each player is countable.

\medskip

The paper is organized as follows.
Section~\ref{section: perliminaries open sets} presents the definitions and preliminary results concerning Gale–Stewart games with open winning sets. In Subsection~\ref{section: perliminaries open sets}.1, we also state the main results that will be proved later in the paper. The connection between games with open winning sets and maximal prefix codes is developed in Section~\ref{Section: max prefix codes}.

Section~\ref{section:Introduction To Nielsen-Schreier Theory} introduces Nielsen–Schreier theory, while the resulting algebraic conditions are established in Section~\ref{section:alg conditions}. The notion of covering is introduced in Section~\ref{section:coverings}, and finally, further discussion about games played with a countable alphabet are presented in Section~{7}.

\section{Preliminary Results for Games with Open Winning Set}
\label{section: perliminaries open sets}
\subsection{Definitions of Games and Main Results}
In this section we give some definitions and useful results for two-player alternating-move games, as given in Solan \cite{Solan}. First, we define a two-player alternating-move game.
\begin{definition}
    A \emph{tree} over an alphabet $\mathcal{A}$ is a nonempty set $T \subseteq \mathcal{A}^{<\N}$
of finite sequences in $\mathcal{A}$ that satisfy the following properties:

$1.$ If $p \in T$ and $p'$ is a prefix of $p$, then $p' \in T$.

$2.$ For every $p \in T$ there is $p'\in T$ such that $p$ is a strict prefix of $p'$.\\
The elements of $\mathcal{A}$ are called \emph{actions}, and an element of T is called a \emph{position}\footnote{Throughout, positions are denoted by angle brackets $\langle\cdot\rangle$.}. An element of $\mathcal{A}^\N$ such that every prefix is a position in $T$ is called a \emph{play}. The set of plays is called the \emph{boundary set} of $T$, and denoted $[T]$. For $p \in T$, we also denote the corresponding sub-tree by $T_p := \{q\in T : p \text{ is a prefix of } q \}.$

\end{definition}
\begin{definition}
    A \emph{two-player alternating-move game} (over an alphabet $\mathcal{A}$) is a pair $(T, W)$, where $T$ is a tree (over $\mathcal{A}$) and $W \subseteq [T]$ is a \emph{winning set}. The game is played as follows. In each stage $n \in \N$:
    \\
    \\
     $\bullet$ If $n$ is even, Player 1 selects an action $a_n \in \mathcal{A}$ such that $\langle a_0, a_1, . . . , a_{n-1}, a_n\rangle \in T$.
     \\
     \\
    $\bullet$ If $n$ is odd, Player 2 selects an action $a_n \in \mathcal{A}$ such that $\langle a_0, a_1, . . . , a_{n-1}, a_n\rangle \in T$.
    \\
    \\
    Player 1 wins if $\langle a_0, a_1, \dots \rangle \in W$, and Player 2 wins otherwise.
    \end{definition}
The tree $T$ describes all possible positions in the game, as well as the set of moves that is available to the players at each position. The set $W$ describes the winning conditions of Player I, namely, all plays that lead to a win of Player I.
\begin{definition}
    A \emph{strategy} $s_1$ of Player 1 is a function that assigns to every position $p \in T$ of even length (including the empty sequence, which has length 0) an action $a = s_1(p) \in \mathcal{A}$ such that the concatenation $p \ \circ \langle a \rangle \in T$. A \emph{strategy} $s_2$ of Player 2 is a function that assigns to every position $p \in T$ of odd length an action $a = s_2(p) \in \mathcal{A}$ such that $p \ \circ \langle a \rangle \in T$.
\end{definition}
Every pair of strategies $s_1,s_2$ of Player 1 and II induce a play $x(s_1,s_2) \in [T]$. A play $x\in [T]$ is “consistent" with strategy $s_1$ if there is some strategy $s_2$ such that $x = x(s_1,s_2)$. 
\begin{definition}
Let $(T, W)$ be a game. A strategy
$s_1 \in S_1(T)$ is a \emph{winning strategy} for Player 1 in $(T,W)$ if all plays $x \in [T]$ consistent with $s_1$ are in $W$. A strategy $s_2 \in S_2(T)$ is \emph{winning} for Player 2 in $(T,W)$ if all plays $x \in [T]$ consistent with $s_2$ are in $W^c$.
We shall sometimes say that a player with a winning strategy \say{can win} the game.
\end{definition}
A game is said to be \emph{determined} if either Player 1 or Player 2 has a winning strategy.

Recall that the collection $\{ [T_p] : p \in T \}$ forms a basis of open sets for the topological space $[T]$ equipped with the product topology. 

A celebrated result of Gale and Stewart~\cite{GaleStewartOpenDET} shows that in every two-player alternating-move game with open winning sets one player has a winning strategy. This was later improved by Martin (1975) to games with borel measurable winning set\footnote{We note that Gale and Stewart's result was originally proved for games with alphabets that are not necessarily finite, a generality that is crucial for Martin's inductive argument~\cite{MartinCovering}.}.
\begin{theorem}
    Let $(T, W)$ be a game. If the set $W$ is
    open in the product topology of $[T]$, then the game $(T, W)$ is determined.
\end{theorem}

The following lemma establishes that, whenever Player~1 is guaranteed to win, they can force the play to enter the winning set within finite number of moves:
\begin{lemma} \label{lemma:win in open mean finite win}
    Player 1  has a winning strategy in the game $(T,W)$ if and only if she has a winning strategy in $(T,W_n = \bigcup_{p\in Z , len(p)\leq n}[T_p])$ for some $n\in \N$.
\end{lemma}
\begin{proof}
    Assume that Player 1  has a winning strategy $s_1 \in S_1(T)$. Notice that $T_{s_1}=\{x(s_1,s_2): s_2 \in S_2(T)\}$ is compact (Tychonoff's theorem), and $W$ is an open covering of it. Thus, there is a finite sub-covering, and in particular there is an $n\in \N$ such that $W_n$ is an open covering of $T_{s_1}$ (which means that $s_1$ is a winning strategy in the game $(T,W_n)$).
\end{proof}

\medskip

We now state some of the main results established in this work. Throughout, we consider the game $(T,W)$, where
\[
T=\mathcal{A}^{\le \mathbb{N}},
\]
and $W$ is an open winning set of the form
\[
W=\bigcup_{p\in Z}[T_p],
\qquad
Z\subseteq T.
\]
Our first result provides a particularly simple sufficient condition for Player~2 to have a winning strategy.

\begin{lemma} \label{Main result: sum condition}
If
\[
\sum_{p \in Z} \frac{1}{|\mathcal{A}|^{\lfloor \mathrm{len}(p)/2 \rfloor}} < 1,
\]
then Player~2 has a winning strategy in $(T,W)$. Moreover, if $|Z| = \infty$, the weaker condition
\[
\sum_{p \in Z} \frac{1}{|\mathcal{A}|^{\lfloor \mathrm{len}(p)/2 \rfloor}} \leq 1
\]
already suffices.
\end{lemma}

We will show that this lemma is, in a sense, parallel to a recent result on the same topic.

\medskip

In Section~\ref{Section: max prefix codes}, we establish an explicit correspondence between open winning sets for which Player~1 has a winning strategy and maximal prefix codes (see Lemma~\ref{lemma:create winning set from max prefix code} and Theorems~\ref{thm:binary-tree-prefix} and~\ref{theorem: max prefix equivalent finite alphabet}). This correspondence is then used to prove Theorem~\ref{corollary: using max prefix codes}, from which the following result follows immediately.

For $p = \langle a_1, \ldots, a_n \rangle \in Z$, define
\[
\hat{p} = \langle a_2, a_4, \ldots, a_{2\lfloor n/2 \rfloor} \rangle,
\]
and let
\[
\hat{Z} = \{ \hat{p} : p \in Z \}.
\]

\begin{theorem}\label{Main result: index property}
Let $F=F(\mathcal{A})$ denote the free group generated by the alphabet $\mathcal{A}$. If the index of the subgroup generated by $\hat{Z}$ is infinite\footnote{We follow standard notation and write $\langle S\rangle$ for the subgroup generated by a set $S\subseteq \mathcal{A}^{\le \mathbb{N}}$.}, that is,
\[
[F:\langle \hat{Z}\rangle]=\infty,
\]
then Player~2 has a winning strategy in the game $(T,W)$.
\end{theorem}

In fact, from Theorem~\ref{corollary: using max prefix codes} follows even a stronger statement: if Player~1 has a winning strategy in $(T,W)$, then there exists a subset
$Z' \subseteq Z$ of cardinality at most
\[
\max_{p \in Z} \left\lfloor \frac{\mathrm{len}(p)}{2} \right\rfloor \cdot (|\mathcal{A}| - 1) + 1
\]
such that
\[
[F : \langle \hat{Z'} \rangle] < \infty.
\]

\medskip

Finally, in Section~\ref{section:coverings}, we combine the notion of coverings with the results of Section~\ref{Section: max prefix codes}, together with preliminary results on open games, to establish the following results concerning maximal prefix codes.

\begin{theorem} \label{main result: max prefix code ineq}
Let $C$ be a maximal prefix code, and let $x = \langle x_1,x_2,\dots\rangle\in \mathcal{A}^{\mathbb{N}}$. Then
\[
\sum_{c = \langle c_1, \ldots, c_n \rangle \in C}
2^{|\{ i \in [n] : c_i \neq x_i \}|}
\cdot \frac{1}{(2|\mathcal{A}| - 1)^{\mathrm{len}(c)}} = 1.
\]
Moreover, if $C$ is a prefix code that is not maximal, then
\[
\sum_{c = \langle c_1, \ldots, c_n \rangle \in C}
2^{|\{ i \in [n] : c_i \neq x_i \}|}
\cdot \frac{1}{(2|\mathcal{A}| - 1)^{\mathrm{len}(c)}} < 1.
\]
\end{theorem}

\subsection{Simple Sufficient Conditions}
The aim of this paper is to develop new methods for characterizing winning sets for which specific Player has a winning strategy, and in particular whether Player 1 possesses a winning strategy in a specified game. For games with open winning sets, we establish a notably simple necessary condition for the existence of such strategy. We further show that this condition sharpens a recent result concerning general winning sets.
\begin{lemma} \label{lemma: necceserry open condition lengths}
    If Player 1 has a winning strategy in the game $(T,W)$ where $T= \mathcal{A}^{\leq \N} , W= \bigcup _ {p \in Z } [T_p] $ is some open winning set. Then $\sum_{p \in Z} \frac{1}{|\mathcal{A}|^{\lfloor {len(p)/2} \rfloor}} \geq 1$.
\end{lemma}
\begin{proof}
    From Lemma~\ref{lemma:win in open mean finite win} We may assume that $Z$ is finite. We prove the claim by induction on \(|Z|\). 

\textbf{Base case (\(|Z|=1\)).}
If \(|Z|=1\), then the assumption that Player~1 has a winning strategy implies that either
\(W=[T_{\langle a\rangle}]\) for some \(a\in\mathcal A\), or \(W=[T]=[T_{\langle\rangle}]\).
In either case, the unique element \(p\in Z\) satisfies \(\lfloor \mathrm{len}(p)/2\rfloor = 0\). Hence
\[
\sum_{p\in Z} \frac{1}{|\mathcal A|^{\lfloor \mathrm{len}(p)/2\rfloor}}
= \frac{1}{|\mathcal A|^{0}} = 1.
\]

\textbf{Inductive step.}
Assume the statement holds for every winning set \(Z'\) with \(|Z'|<n\).
Let \(Z\) be a winning set of size \(n\), and suppose Player~1 has a winning strategy in the game \((T,W)\) with
\[
W = \bigcup_{p\in Z} [T_p], \qquad |Z| = n.
\]

Let \(a_1\in\mathcal A\) be Player~1 first move according to this strategy.
Because the strategy is winning, for every reply \(a_2\in\mathcal A\) of Player~2, Player~1 still has a winning strategy in the subgame beginning at the position \(\langle a_1,a_2\rangle\).
For each such subgame, the set
\[
Z_{a_2} := Z \cap T_{\langle a_1,a_2\rangle}
\]
is strictly smaller than \(Z\).
Hence, by the induction hypothesis,
\[
\sum_{p\in Z_{a_2}}
   \frac{1}{|\mathcal A|^{\lfloor (\mathrm{len}(p)-2)/2 \rfloor}}
   \ge 1
   \qquad \text{for every } a_2\in\mathcal A.
\]

We now sum over all possible replies \(a_2\):
\[
\sum_{p\in Z}
   \frac{1}{|\mathcal A|^{\lfloor \mathrm{len}(p)/2\rfloor}}
   \ge
   \sum_{p\in Z\cap T_{\langle a_1\rangle}}
   \frac{1}{|\mathcal A|^{\lfloor \mathrm{len}(p)/2\rfloor}}
\]
\[
=
\sum_{a_2\in\mathcal A}
   \sum_{p\in Z_{a_2}}
   \frac{1}{|\mathcal A|^{\lfloor \mathrm{len}(p)/2\rfloor}}
=
\sum_{a_2\in\mathcal A}
   \sum_{p\in Z_{a_2}}
   \frac{1}{|\mathcal A|^{\lfloor (\mathrm{len}(p)-2)/2\rfloor}}
   \cdot
   \frac{1}{|\mathcal A|}
\]
\[
\ge_{\text{Ind.\ hyp.}}
\sum_{a_2\in\mathcal A}
   \frac{1}{|\mathcal A|}
= 1.
\]

This completes the induction and the proof.
\end{proof}

\medskip

Lemma~\ref{Main result: sum condition} follows immediately from Lemma~\ref{lemma: necceserry open condition lengths}.

\medskip

\begin{proof}[Proof of Lemma~\ref{Main result: sum condition}]
The first assertion follows directly from Lemma~\ref{lemma: necceserry open condition lengths}. 
If $|Z|=\infty$, we extend the argument using Lemma~\ref{lemma:win in open mean finite win}.

Suppose, toward a contradiction, that Player~1 has a winning strategy. Then there exists a finite subset $Z'\subseteq Z$ such that Player~1 has a winning strategy in the game with winning set
\[
W'=\bigcup_{p'\in Z'} [T_{p'}].
\]
However, for any finite subset $Z'\subseteq Z$, we have
\[
\sum_{p'\in Z'} \frac{1}{2^{\lfloor \mathrm{len}(p')/2 \rfloor}} < 1,
\]
and hence Player~2 has a winning strategy in $(T,W')$, a contradiction. This completes the proof.
\end{proof}

We also observe that if Player~1 has a winning strategy, then there exists an infinite descending sequence of sub-trees, corresponding to a sequence of subgames, on which Player~1 retains a winning strategy. Consequently, Lemma~\ref{lemma: necceserry open condition lengths} admits a natural extension.

\begin{corollary}
Let $T=\mathcal A^{\leq \mathbb N}$ and $W=\bigcup_{p\in Z}[T_p]$ be an open winning set.  
If Player~1 has a winning strategy in the game $(T,W)$, then for every $0 \le n\le \min_{p\in Z}{\mathrm{len(p)}}$, 
\[
\max_{\substack{p^{*}\in T \\ \mathrm{len}(p^{*})=n}}
\;
|\mathcal A|^{\lfloor n/2 \rfloor}
\sum_{p\in Z\cap T_{p^{*}}}
\frac{1}{|\mathcal A|^{\lfloor \mathrm{len}(p)/2 \rfloor}}
\ge 1.
\]
\end{corollary}

We now observe that games with a general open winning set can be reduced to a simpler, canonical case.

\begin{lemma} \label{Lemma: can take minimal size}
Let $T = \mathcal A^{\leq \mathbb N}$ and let 
\[
W = \bigcup_{p\in Z} [T_p]
\] 
be an open winning set. If Player~1 has a winning strategy in $(T,W)$, then there exists a subset $Z'\subseteq Z$ such that Player~1 also has a winning strategy in $(T,\bigcup_{p\in Z'} [T_p])$, and
\[
\sum_{p\in Z'} \frac{1}{|\mathcal A|^{\lfloor \mathrm{len}(p)/2 \rfloor}} = 1.
\]
\end{lemma}

\begin{proof}
The proof is by induction on $|Z|$.  

\textbf{Base case ($|Z|=1$).}  
If $|Z|=1$, then $Z$ contains only a single position of length $1$ or $0$, and the desired property follows immediately.

\textbf{Inductive step.}  
Assume the statement holds for all sets of size less than $|Z|$.  
Since Player~1 has a winning strategy, there exists $a_0\in \mathcal A$ such that she has a winning strategy in the subgames $T_{\langle a_0,a\rangle}$ for all $a\in\mathcal A$.  
By the induction hypothesis, for each $a\in\mathcal A$ there exists $Z_a'\subseteq Z\cap T_{\langle a_0,a\rangle}$ such that Player~1 has a winning strategy in the corresponding game on $T_{\langle a_0,a \rangle }$, and
\[
\sum_{p\in Z_a'} \frac{1}{|\mathcal A|^{\lfloor \mathrm{len}(p)/2 \rfloor}} = \frac{1}{|\mathcal A|}.
\]  
Taking 
\[
Z' := \bigcup_{a\in\mathcal A} Z_a',
\]
we then have
\[
\sum_{p\in Z'} \frac{1}{|\mathcal A|^{\lfloor \mathrm{len}(p)/2 \rfloor}} = \sum_{a\in\mathcal A} \sum_{p\in Z_a'} \frac{1}{|\mathcal A|^{\lfloor \mathrm{len}(p)/2 \rfloor}} = 1,
\]
as desired.
\end{proof}

\medskip

In light of Lemma~\ref{Lemma: can take minimal size}, we make the following definition. It will become crucial in Section~\ref{Section: max prefix codes}.

\begin{definition} \label{definition: minimal size}
An open set 
\[
W = \bigcup_{p\in Z} [T_p] \subseteq \mathcal A^{\mathbb N}
\] 
is said to be \emph{minimal-size} if
\[
\sum_{p\in Z} \frac{1}{|\mathcal A|^{\lfloor \mathrm{len}(p)/2 \rfloor}} = 1.
\]  
Similarly, we say that $Z$ (the set of positions) is \emph{minimal-size} if $\bigcup_{p\in Z}[T_p]$ is minimal-size.
\end{definition}

\begin{remark}
The same argument as in the proof of Lemma~\ref{Main result: sum condition}, shows that given $T=\mathcal A^{\leq \mathbb N}$ and $W=\bigcup_{p\in Z}[T_p]$ an open set such that $Z$ is \emph{infinite}. If $W$ is of minimal-size, then Player~2 has a winning strategy in the game $(T,W)$.
\end{remark}

From now on, we may assume without loss of generality that $W$ is minimal-size, because otherwise Lemma~\ref{Lemma: can take minimal size} provides a minimal-size subset of $Z$ for which the same player has a winning strategy. We further may assume that $Z$ is finite, as otherwise it would mean that Player~2 has a winning strategy. 

\subsection{Redefining Open Winning Set: Hausdorff Dimension Reduction}
In this work, our primary focus is on games with open winning sets. The two (very simple) lemmas presented below allow us to treat open winning sets in a slightly different framework, which will enable us to apply recent necessary conditions involving Hausdorff dimension.

We begin by considering a game $(T,W)$ over a finite alphabet $\mathcal{A}$, where $T=\mathcal{A}^{<\mathbb{N}}$. Suppose that $W$ is open in the product topology. Then there exists a subset $Z\subseteq T$ such that
\[
W=\bigcup_{p\in Z}[T_p].
\]

Since for any $p,q\in T$ we have
\[
[T_p]\cap[T_q]\neq\emptyset
\quad\text{if and only if}\quad
[T_p]\subseteq[T_q]\ \text{or}\ [T_q]\subseteq[T_p],
\]
we may assume without loss of generality that $Z$ is \emph{minimal}, in the sense that no element of $Z$ is a prefix of another. Note, that any set $Z$ of minimal size in the sense of Definition~\ref{definition: minimal size}, where Player~1 can win in the corresponding game, is also minimal in the sense that no position in $Z$ is a prefix of another.

We first show that, without loss of generality, the positions in $Z$ may be assumed to have even length. Throughout, concatenation is denoted by the symbol $\circ$.

\begin{lemma} \label{lemma: can ssume even lengths}
Given an open winning set 
\[
W = \bigcup_{p\in Z} [T_p],
\] 
we have
\[
W = 
\bigcup_{\substack{p\in Z \\ \mathrm{len}(p)\ \text{even}}} [T_p] 
\;\cup\;
\bigcup_{\substack{p\in Z \\ \mathrm{len}(p)\ \text{odd}}} 
\bigcup_{a \in \mathcal{A}} [T_{p \circ \langle a \rangle}].
\]
\end{lemma}

\begin{proof}
Indeed, for any $x \in [T]$,
\[
x \in W \iff \exists p \in Z \text{ such that } p \text{ is a prefix of } x
\]
\[
\iff \exists p \in Z,\, \exists a \in \mathcal{A} \text{ such that } p \circ \langle a \rangle \text{ is a prefix of } x.
\]
\end{proof}

\medskip

From now on, whenever we consider games with open winning sets, we may assume without loss of generality that all positions in $Z$ are of even length.

\medskip

Next, given an open set, we define an auxiliary game such that the player with a winning strategy is the same in both games.  
Denote the submonoid with respect to concatenation generated by a set $P \subseteq T$ by
\[
P _{\mathrm{concat}} = \{ p_1 \circ p_2 \circ p_3 \circ \cdots : p_i \in P \}.
\]

\begin{lemma} \label{lemma: concat is open}
Assume that every position $p \in Z$ has even length.  
Player~1 can win the game $(T,W)$ if and only if she can win the game $(T,Z_{\mathrm{concat}})$.
\end{lemma}

\begin{proof}
Since $Z_{\mathrm{concat}} \subseteq W$, the first direction is immediate.  

Conversely, because all positions in $Z$ have even length, Player~1 can repeat her winning strategy from the game $(T,W)$ every time she reaches a position in $Z$, thereby guaranteeing a win in $(T,Z_{\mathrm{concat}})$.  
\end{proof}

\medskip

By a slight abuse of notation, we will write $Z_{\mathrm{concat}}$ even when the positions in $Z$ are not of even length.  
This indicates that we first apply Lemma~\ref{lemma: can ssume even lengths} to transform $Z$ into a set of even-length positions and then take the submonoid generated by the resulting set.

A recent result of Bella{\"i}che and Rosenzweig \cite{DETwinnerHDIM} establishes an additional necessary condition for Player~1 to possess a winning strategy\footnote{This holds in a much broader setting than the open case, and in particular for general winning sets and infinite alphabet.}.

\begin{theorem}
\label{theorem: hausdorf 1/2}
If Player~1 has a winning strategy in the game $(T,W)$, where $T=\mathcal A^{\leq \mathbb N}$, then the Hausdorff dimension of $W$ satisfies
\[
\dim_H(W) \ge \tfrac12 .
\]
\end{theorem}

At first glance, applying this condition to an \emph{open} winning set may appear redundant, since every open set in the product topology is known to have Hausdorff dimension~$1$.  
However, this obstacle is avoided by instead considering the equivalent winning set defined in Lemma~\ref{lemma: concat is open}, which we denote by $Z_{\mathrm{concat}}$.  
We will show that Theorem~\ref{theorem: hausdorf 1/2} is in fact a consequence of Lemma~\ref{lemma: necceserry open condition lengths} when restricted to open winning sets.  
Moreover, in this special case, the two results are equivalent.

\begin{lemma} \label{lemma: Housdorf and sum 1/2}
Let $T=\mathcal A^{\leq \mathbb N}$ and let $W=\bigcup_{p\in Z}[T_p]$ be an open set, with $Z \subseteq T$ being a finite collection of even positions.  
Then $\dim_H(Z_{\mathrm{concat}})<\tfrac12$, if and only if
\[
\sum_{p\in Z} \frac{1}{|\mathcal A|^{ \mathrm{len}(p)/2 }} < 1.
\]
\end{lemma}
Before giving the proof, we state a result that will be used below. This statement is a corollary of the Hausdorff dimension game introduced by Das, Fishman, Simmons, and Urbański~\cite{DasFishmanHousdorfGame}; see also~\cite{DETwinnerHDIM} for an exposition in the dyadic setting.

\begin{theorem}\label{theorem Fishman game housdorf}
Let \(T=\mathcal{A}^{\le \mathbb{N}}\) and let \(S\subseteq [T]\) be a Borel set. 
If there exist integers \(k\ge 1\) and a set \(A\subseteq \mathcal{A}^k\) with \(|A|\ge |\mathcal{A}|^{k d}\) and \(A_{\mathrm{concat}} \subseteq S\),
then \(\dim_H(S)\ge d\).
\end{theorem}

\begin{proof}[Proof of Lemma~\ref{lemma: Housdorf and sum 1/2}]
We recall the definition of Hausdorff dimension.

\smallskip

Let \(\delta>0\). A \(\delta\)-cover of \(Z_{\mathrm{concat}}\) is a countable collection \(\{U_i\}_{i\in I}\) such that \(Z_{\mathrm{concat}}\subseteq\bigcup_{i\in I}U_i\) and \(\operatorname{diam}(U_i)<\delta\) for every \(i\in I\). The \(d\)-dimensional Hausdorff measure of \(Z_{\mathrm{concat}}\) is
\[
H^d(Z_{\mathrm{concat}})
=
\liminf_{\delta\to 0}
\left\{
\sum_{i\in I}(\operatorname{diam} U_i)^d
\;\middle|\;
\{U_i\}_{i\in I}\text{ is a }\delta\text{-cover of }Z_{\mathrm{concat}}
\right\},
\]
and the Hausdorff dimension is
\[
\dim_H(Z_{\mathrm{concat}})
=
\inf\{d\ge 0:\; H^d(Z_{\mathrm{concat}})=0\}.
\]

Let \(2n=\max_{p\in Z}\operatorname{len}(p)\), and fix \(p\in Z\). We replace \(p\) by a uniformly expanded family of extensions of length \(2n\), constructed by inserting a fixed symbol \(a_0\in\mathcal A\) in every even position and allowing every symbol of \(\mathcal A\) to appear in the intervening odd positions.

Recall that we have assumed every \(p\in Z\) has even length. Define
\[
\operatorname{Extend}_n(p)
:=
\Bigl\{
p \circ \langle a_0,a_1\rangle \circ \langle a_0,a_2\rangle \circ \cdots \circ
\langle a_0,a_{\,n-(\operatorname{len}(p)/2)}\rangle
:
a_1,\dots,a_{\,n-(\operatorname{len}(p)/2)}\in\mathcal A
\Bigr\}.
\]
Every \(q\in\operatorname{Extend}_n(p)\) satisfies \(\operatorname{len}(q)=2n\), and
\[
\bigl|\operatorname{Extend}_n(p)\bigr|
=|\mathcal A|^{\,n-(\operatorname{len}(p)/2)}.
\]

Set
\[
W_n := \left( \bigcup_{p\in Z} \operatorname{Extend}_n(p)\right)_{\mathrm{concat}}.
\]
We claim that if \(\dim_H(Z_{\mathrm{concat}})<1/2\) then \(\dim_H(W_n)<1/2\). To see this, let \(\{U_i\}_{i\in I}\) be a \(\delta\)-cover of \(Z_{\mathrm{concat}}\). We may assume each \(U_i\) is a basic open set: if \(\operatorname{diam}(U_i)=0\) then \(U_i\) contains a single point and can be replaced by a basic open set of arbitrarily small diameter; moreover,
\[
\operatorname{diam}(U_i)=|\mathcal A|^{-\min\{k : \exists x,y\in U_i,\; x_k\neq y_k\}},
\]
so replacing \(U_i\) by a containing basic open set does not increase its diameter.

Since the infimum defining \(H^d\) may be taken over basic open covers, we may further assume each \(U_i\) has the form \([T_p]\), where \(p\) is a prefix of a concatenation \(p_1\circ\cdots\circ p_k\) with \(p_1,\dots,p_k\in Z\). For such a basic open set \(U_i=[T_p]\) with
\(p=p_1\circ\cdots\circ p_{k-1}\circ p'\) and \(p'\preceq p_k\in Z\), define
\[
\operatorname{Extend}_n(U_i)
:=\bigl\{[T_q]:
q=q_1\circ\cdots\circ q_{k-1}\circ p',\;
q_j\in\operatorname{Extend}_n(p_j)\ \text{for }1\le j\le k-1\bigr\}.
\]
Then \(\bigcup_{i\in I}\operatorname{Extend}_n(U_i)\) is an open cover of \(W_n\).

Because \(\dim_H(Z_{\mathrm{concat}})<1/2\), there exists \(\epsilon_0>0\) such that \(H^{1/2-\epsilon_0}(Z_{\mathrm{concat}})=0\). Put \(\epsilon=\epsilon_0/(2n)\). We claim that
\[
\sum_{i\in I}\sum_{U'\in\operatorname{Extend}_n(U_i)}(\operatorname{diam} U')^{\,1/2-\epsilon}
\le
\sum_{i\in I}(\operatorname{diam} U_i)^{\,1/2-\epsilon_0}.
\]
To verify this inequality, fix \(U_i=[T_p]\) with
\(p=p_1\circ\cdots\circ p_{k-1}\circ p'\) as above. Then
\[
\sum_{U'\in\operatorname{Extend}_n(U_i)}(\operatorname{diam} U')^{\,1/2-\epsilon}
=
\sum_{\substack{q_j\in\operatorname{Extend}_n(p_j)\\ j=1,\dots,k-1}}
\biggl(\frac{1}{|\mathcal A|}\biggr)^{(2n(k-1)+\operatorname{len}(p'))(1/2-\epsilon)}.
\]
Since there are \(|\mathcal A|^{\,n-\operatorname{len}(p_j)/2}\) choices for each \(q_j\), the preceding display equals
\[
\biggl(\frac{1}{|\mathcal A|}\biggr)^{(2n(k-1)+\operatorname{len}(p'))(1/2-\epsilon)}
\prod_{j=1}^{k-1} |\mathcal A|^{\,n-\operatorname{len}(p_j)/2}.
\]
A short rearrangement yields
\[
\sum_{U'\in\operatorname{Extend}_n(U_i)}(\operatorname{diam} U')^{\,1/2-\epsilon}
=
|\mathcal A|^{\,\epsilon(2n(k-1)+\operatorname{len}(p'))}\cdot
\bigl(\operatorname{diam} U_i\bigr)^{1/2}.
\]
Since \(\epsilon(2n(k-1)+\operatorname{len}(p'))\le \epsilon_0\operatorname{len}(p)\) and \((\operatorname{diam} U_i)^{\epsilon_0} =  |\mathcal{A}|^{-\epsilon_0 \operatorname{len}(p)} \), we obtain
\[
\sum_{U'\in\operatorname{Extend}_n(U_i)}(\operatorname{diam} U')^{\,1/2-\epsilon}
\le (\operatorname{diam} U_i)^{1/2-\epsilon_0},
\]
as claimed.

It follows that
\[
H^{1/2-\epsilon}(W_n)
\le
\liminf_{\delta\to 0}
\Biggl\{
\sum_{i\in I}\sum_{U'\in\operatorname{Extend}_n(U_i)}(\operatorname{diam} U')^{\,1/2-\epsilon}
\;\Bigm|\;
\{U_i\}_{i\in I}\ \text{is an open }\delta\text{-cover of }Z_{\mathrm{concat}}
\Biggr\}
\]
\[
\le
\liminf_{\delta\to 0}
\Biggl\{
\sum_{i\in I}(\operatorname{diam} U_i)^{\,1/2-\epsilon_0}
\;\Bigm|\;
\{U_i\}_{i\in I}\ \text{is an open }\delta\text{-cover of }Z_{\mathrm{concat}}
\Biggr\}
=0.
\]
Hence \(\dim_H(W_n)<1/2\).

\medskip

Denote
\[
A:=\bigcup_{p\in Z}\operatorname{Extend}_n(p)\subseteq\mathcal A^{2n}.
\]
Note that
\[
|A|=\sum_{p\in Z} |\mathcal A|^{\,n-(\operatorname{len}(p)/2)}.
\]
Since \(A_{\mathrm{concat}}\subseteq W_n \) and \(\dim_H(W_n)<1/2\), Theorem~\ref{theorem Fishman game housdorf} implies \(|A|<|\mathcal A|^n\). Therefore
\[
\sum_{p\in Z} \frac{1}{|\mathcal A|^{\,\operatorname{len}(p)/2}}<1,
\]
which is the desired inequality.

\medskip

For the converse, assume
\[
\sum_{p\in Z}\biggl(\frac{1}{|\mathcal A|}\biggr)^{d\cdot\operatorname{len}(p)}<1
\quad\text{for }d=\tfrac12.
\]
By continuity in the exponent \(d\), there exists \(\epsilon>0\) such that
\[
\sum_{p\in Z}\biggl(\frac{1}{|\mathcal A|}\biggr)^{(1/2-\epsilon)\operatorname{len}(p)}<1.
\]
For each \(n\ge 1\) the family
\[
\bigcup_{(p_1,\dots,p_n)\in Z^n} [T_{p_1\circ\cdots\circ p_n}]
\]
is an open cover of \(Z_{\mathrm{concat}}\), and
\[
\max_{(p_1,\dots,p_n)\in Z^n}\operatorname{diam}([T_{p_1\circ\cdots\circ p_n}])
=|\mathcal A|^{-n\cdot \min_{p\in Z}\operatorname{len}(p)} \underset{n \to \infty}{\to}0.
\]
Hence
\[
H^{1/2-\epsilon}(Z_{\mathrm{concat}})
\le
\lim_{n\to\infty}
\sum_{(p_1,\dots,p_n)\in Z^n}
\operatorname{diam}([T_{p_1\circ\cdots\circ p_n}])^{\,1/2-\epsilon}
\]
\[
=
\lim_{n\to\infty}
\sum_{(p_1,\dots,p_n)\in Z^n}
\biggl(\frac{1}{|\mathcal A|}\biggr)^{(1/2-\epsilon)\sum_{j=1}^n\operatorname{len}(p_j)}
=
\lim_{n\to\infty}
\Biggl(
\sum_{p\in Z}\biggl(\frac{1}{|\mathcal A|}\biggr)^{(1/2-\epsilon)\operatorname{len}(p)}
\Biggr)^{\!n}.
\]
By the assumption the last expression vanishes, therefore \(H^{1/2-\epsilon}(Z_{\mathrm{concat}})=0\) and \(\dim_H(Z_{\mathrm{concat}})<1/2\).
\end{proof}

\begin{example}\label{example: infinite minimal size is win for 2}
Consider the game played on the tree $\{0,1\}^{\le \mathbb{N}}$ with open winning set
\[
W=\bigcup_{p\in Z}[T_p],
\qquad
\text{where }
Z:=\Bigl\{\underbrace{\langle 0,1\rangle \circ \langle 0,1\rangle \circ \cdots \circ \langle 0,1\rangle}_{k\text{ times}} \circ \langle 0,0\rangle : k\in \mathbb{N}\Bigr\}.
\]
Equivalently, $Z$ consists of all even-length positions in which all even-indexed entries except the last are equal to $1$, while all remaining entries are equal to $0$. Observe that
\[
\sum_{p\in Z} \frac{1}{2^{\lfloor \mathrm{len}(p)/2 \rfloor}}
=\sum_{k=1}^{\infty} \frac{1}{2^k}
=1.
\]
Therefore, by Lemma~\ref{Main result: sum condition}, Player~2 has a winning strategy in this game. Indeed, such a strategy is given by always choosing the action $1$ on her turn.

One may also verify that
\[
Z_{\mathrm{concat}}
=
\Bigl\{ x=\langle x_0,x_1,\ldots\rangle : x_{2k}=0 \text{ for all } k\in\mathbb{N} \Bigr\}
\setminus
\{\langle 0,1,0,1,0,1,\ldots\rangle\}.
\]
Consequently, $\dim_H(Z_{\mathrm{concat}})=\tfrac12$, and a similar conclusion can be reached by exploiting directly the fact that $Z$ is infinite.
\end{example}

\section{Characterization via Maximal Prefix Codes} \label{Section: max prefix codes}

We introduce an alternative viewpoint for the study of games with open winning sets.
Prefix codes (also known as \emph{prefix-free} or \emph{instantaneous} codes) are a classical object of study in information theory and theoretical computer science. In this section we establish a precise connection between prefix codes and two-player alternating-move games with open winning sets. In particular, we show that maximal prefix codes correspond exactly to open winning sets for which Player~1 has a unique winning strategy. 

\begin{definition}[Prefix codes and maximal prefix codes] \label{def of prefix codes}
A \emph{prefix code} over a finite alphabet $\mathcal{A}$ is a set $C$ of finite sequences (or \say{words}) over $\mathcal{A}$ such that no word in $C$ is a prefix of another. A prefix code $C$ is said to be \emph{maximal} if there is no prefix code $C'$ over the same alphabet with $C \subsetneq C'$.
\end{definition}

The next lemma formalizes the connection between maximal prefix codes and open winning sets in which Player~1 has exactly one winning strategy. Fixing a strategy for Player~1, we show that a maximal prefix code naturally induces a minimal-size (in the sense of Definition ~\ref{definition: minimal size}) open winning set. Conversely, every such winning set arises in this way from a maximal prefix code.

\medskip

Fix a strategy $s_1$ of Player~1 in the tree $\mathcal{A}^{\le \N}$. For any set $C \subseteq \mathcal{A}^{\leq \N}$, define

\medskip 

\[Z(C)= Z_{s_1}(C) := \Bigl\{\,  p = \langle a_1, c_1, a_2, c_2, \dots, a_n, c_n \rangle : \langle c_1, \dots, c_n \rangle \in C, a_1,\dots ,a_n \in \mathcal{A}\]
\[\text{ are the actions prescribed to Player~1 by the strategy } s_1 \,\Bigr\}.
\]

\begin{lemma} \label{lemma:create winning set from max prefix code}
A set $C$ is a maximal prefix code if and only if the following two conditions hold:
\begin{enumerate}
\item $Z_{s_1}(C)$ is a minimal-size set, and
\item the strategy $s_1$
is (the only) winning strategy for Player~1 in the game $(T,W)$, where
\[
T = \mathcal{A}^{\leq \mathbb N}
\qquad \text{and} \qquad
W = \bigcup_{p \in Z_{s_1}(C)} [T_p].
\]
\end{enumerate}
\end{lemma}
\begin{proof}
We prove the two directions separately.

\medskip
\noindent
\emph{($\Rightarrow$)}  
Assume that $C$ is a maximal prefix code. We claim that the strategy $s_1$ of Player~1 is a winning strategy in the game $(T,W)$.

Indeed, fix an arbitrary strategy $s_2$ for Player~2. Write the resulting play in the form
\[
x(s_1,s_2)=\langle a_1,x_1,a_2,x_2,\dots\rangle .
\]
By Lemma~3 of \cite{maxprefixcodeiswinning}, there exists a word
\[
c=\langle c_1,\dots,c_k\rangle \in C
\]
which is a prefix of $\langle x_1,x_2,\dots\rangle$. Consequently,
\[
\langle a_1,c_1,a_2,c_2,\dots,a_k,c_k\rangle \in Z(C)
\]
is a prefix of the play $x(s_1,s_2)$, showing that $x(s_1,s_2)\in W$.

To see that $Z(C)$ is of minimal size, recall from Lemma~\ref{lemma: necceserry open condition lengths} that any open winning set $Z$ must satisfy
\[
\sum_{p\in Z} \frac{1}{|\mathcal{A}|^{\lfloor \mathrm{len}(p)/2 \rfloor}} \geq 1.
\]
On the other hand, by Kraft’s inequality for prefix codes (see Theorem~1 in \cite{maxprefixcodeiswinning}),
\[
\sum_{p\in Z(C)} \frac{1}{|\mathcal{A}|^{\lfloor \mathrm{len}(p)/2 \rfloor}}
= \sum_{c\in C} \frac{1}{|\mathcal{A}|^{\mathrm{len}(c)}} \leq 1.
\]
Combining the two inequalities, we conclude that
\[
\sum_{p\in Z(C)} \frac{1}{|\mathcal{A}|^{\lfloor \mathrm{len}(p)/2 \rfloor}} = 1,
\]
and hence $Z(C)$ is minimal-size set.

\medskip
\noindent
\emph{($\Leftarrow$)}  
Assume that Player~1 has a winning strategy $s_1$ in $(T,W)$ and that
\[
\sum_{p\in Z(C)} \frac{1}{|\mathcal{A}|^{\lfloor \mathrm{len}(p)/2 \rfloor}} = 1.
\]

We first show that $C$ is a prefix code. Suppose otherwise that there exist distinct words $c^1,c^2\in C$ such that $c^1$ is a prefix of $c^2$. Let $p^1,p^2\in Z(C)$ denote the corresponding positions. Then $p^1$ is a prefix of $p^2$, and therefore
\[
[T_{p^2}] \subseteq [T_{p^1}].
\]
It follows that
\[
W=\bigcup_{p\in Z(C)} [T_p]
   =\bigcup_{p\in Z(C)\setminus\{p^2\}} [T_p].
\]
Since $s_1$ is a winning strategy, Lemma~\ref{lemma: necceserry open condition lengths} implies
\[
\sum_{p\in Z(C)\setminus\{p^2\}}
\frac{1}{|\mathcal{A}|^{\lfloor \mathrm{len}(p)/2 \rfloor}} \geq 1,
\]
which contradicts the assumption
\[
\sum_{p\in Z(C)} \frac{1}{|\mathcal{A}|^{\lfloor \mathrm{len}(p)/2 \rfloor}} = 1.
\]
Thus $C$ must be a prefix code.

It remains to show that $C$ is maximal. Suppose there exists a word
$c^\ast\in \mathcal{A}^{\leq\mathbb{N}}\setminus C$ such that
$C\cup\{c^\ast\}$ is still a prefix code. Define a strategy $s_2$ for Player~2 by
\[
s_2(p)=
\begin{cases}
c^\ast_{\frac{\mathrm{len}(p)-1}{2}}
& \text{if } \mathrm{len}(p)\leq 2\,\mathrm{len}(c^\ast)+1,\\[4pt]
\text{some } a \in \mathcal{A}  & \text{otherwise}.
\end{cases}
\]
Then the resulting play $x(s_1,s_2)$ cannot have any prefix in $Z(C)$: otherwise, for some $c\in C$, either $c$ would be a prefix of $c^\ast$ or $c^\ast$ a prefix of $c$, contradicting the assumption that $C\cup\{c^\ast\}$ is a prefix code. This contradicts the fact that $s_1$ is winning, and hence $C$ must be maximal.
\end{proof}

\begin{remark}
    The preceding lemma establishes a precise correspondence between maximal prefix codes and open winning sets for which Player~1 has a unique winning strategy. More precisely, every open winning set admitting exactly one winning strategy arises from a maximal prefix code, and conversely, every maximal prefix code determines a family of such winning sets.
    This connection motivates the characterization that we now present, and will be used in the proof of theorem~\ref{Theorem: Kraft like inequality with coverings}.
\end{remark}

We restrict attention to games played on the full binary tree and consider open winning sets. Our goal is to give a complete and explicit criterion for the existence of a winning strategy for Player~1. The guiding idea is that a strategy for Player~1 should induce, from the set of positions consistent with that strategy, a maximal prefix code. The following construction provides a convenient mechanism for detecting this property.

\medskip
Let $x=\langle x_1,x_2,\dots\rangle \in \{0,1\}^{\mathbb{N}}$ and let
$p=\langle a_1,\dots,a_n\rangle \in \{0,1\}^{\leq \mathbb{N}}$ be a finite position. Define
\[
c_x(p)
  = \langle c_x^1(p),\dots,c_x^{\lfloor \mathrm{len}(p)/2 \rfloor}(p)\rangle,
\]
where
\[
c_x^k(p) := x_k \cdot a_{2k-1} + a_{2k} \pmod{2}.
\]
In words, for a fixed choice of $x$, the map $c_x$ assigns to each position in the binary tree a position in the binary tree of half the length, in a manner that coherently encodes all positions consistent with a fixed strategy of Player~1. While this construction is somewhat artificial, it serves an important purpose: positions that are inconsistent with a single strategy of Player~1 are mapped to the same words and thus prevent the code being maximal prefix code.

\medskip
We may now state the main result of this section.
\begin{theorem}
\label{thm:binary-tree-prefix}
Player~1 has a winning strategy in the game
\[
(T,W), \qquad T=\{0,1\}^{\leq\mathbb{N}}, \quad W=\bigcup_{p\in Z}[T_p],
\]
where $Z$ is a minimal-size set of positions, if and only if for every
$x\in\{0,1\}^{\mathbb{N}}$ the set
\[
C_x(Z):=\{\, c_x(p) : p\in Z \,\}
\]
is a maximal prefix code.
\end{theorem}
\begin{proof}
Assume first that Player~1 has a winning strategy $s_1$.  
We claim that for every $x \in \{0,1\}^{\mathbb{N}}$, the set
$C_x(Z_{s_1})$ is a maximal prefix code, where $Z_{s_1}$ denotes the set
of positions in $Z$ that are consistent with the strategy $s_1$.

We begin by showing that $C_x(Z_{s_1})$ is a prefix code.
Suppose, to the contrary, that there exist
$p=\langle a_1,\dots,a_m\rangle$ and
$p'=\langle a'_1,\dots,a'_n\rangle$ in $Z_{s_1}$ such that
$c_x(p)$ is a prefix of $c_x(p')$. Then for every
$k \leq \lfloor m/2 \rfloor$ we have
\[
x_k \cdot a_{2k-1} + a_{2k}
\equiv
x_k \cdot a'_{2k-1} + a'_{2k}
\pmod{2}.
\]
We claim that $p$ must be a prefix of $p'$. Otherwise, let $i$ be the
minimal index such that $a_i \neq a'_i$.

\medskip
\noindent
\textbf{Case 1:} $i=2k$ for some $k \leq \lfloor m/2 \rfloor$.  
Then $a_{2k-1}=a'_{2k-1}$ and
\[
x_k \cdot a_{2k-1} + a_{2k}
\neq
x_k \cdot a'_{2k-1} + a'_{2k},
\]
contradicting the equality above.

\medskip
\noindent
\textbf{Case 2:} $i=2k-1$ for some $k \leq \lfloor m/2 \rfloor$.  
This case is impossible, since $a_{2k-1}$ and $a'_{2k-1}$ are both moves
chosen by Player~1 at the same position
$\langle a_1,\dots,a_{2k-2}\rangle=\langle a'_1,\dots,a'_{2k-2}\rangle$,
and therefore must coincide.

\medskip
In all cases we conclude that $p$ is a prefix of $p'$, contradicting the
assumption that $Z$ (and hence $Z_{s_1}$) is minimal. Therefore,
$C_x(Z_{s_1})$ is a prefix code.

Since, by assumption $Z$ is a minimal-size open winning set, it follows that $Z_{s_1}=Z$.
By Kraft’s inequality, we conclude that $C_x(Z)$ is in fact a
\emph{maximal} prefix code.

\medskip
\noindent
For the converse direction, assume that for every
$x \in \{0,1\}^{\mathbb{N}}$, the set $C_x(Z)$ is a maximal prefix code.
We define a strategy $s_1$ for Player~1 as follows. For every even-length
position $p \in T$, set
\[
s_1(p) =
\begin{cases}
a \in \{0,1\}, & \text{if both } p \circ \langle a,0\rangle
\text{ and } p \circ \langle a,1\rangle \text{ are prefixes of elements of } Z,\\
0, & \text{if no such } a \text{ exists}.
\end{cases}
\]
This strategy is well defined.

We claim that $s_1$ is a winning strategy for Player~1.
Fix any strategy $s_2$ of Player~2 and consider the induced play
$x(s_1,s_2)$. We show that at every even position along this play, there
exists an action $a \in \{0,1\}$ satisfying the first condition of
$s_1$ above---thus proving it is indeed a winning strategy.

Suppose, to the contrary, that there exists a strategy $s_2$ and
a minimal $k$ such that for
\[
p := x(s_1,s_2)_{[2k]}
\]
no such $a$ exists. This means that the only prefixes of elements of $Z$
extending $p$ are\footnote{At most}
\[
p \circ \langle 0, s_2(p \circ \langle 0\rangle)\rangle
\quad \text{and} \quad
p \circ \langle 1, s_2(p \circ \langle 1\rangle)\rangle.
\]

Define $x \in \{0,1\}^{\mathbb{N}}$ by
\[
x_k = s_2(p \circ \langle 0\rangle) + s_2(p \circ \langle 1\rangle) \pmod{2},
\qquad
x_j = 0 \ \text{for } j \neq k.
\]
Then for every $p' \in Z$ extending $p$, we have
\[
c_x^k(p') \in
\Bigl\{
x_k \cdot 0 + s_2(p \circ \langle 0\rangle),
\;
x_k \cdot 1 + s_2(p \circ \langle 1\rangle)
\Bigr\}
=
\{ s_2(p \circ \langle 0\rangle) \}.
\]
Hence all elements of $C_x(Z)$ extending $c_x(p)$ agree on their
$k$-th coordinate. Consequently,
\[
C_x(Z) \cup \{ c_x(p) \circ \langle 1 - s_2(p \circ \langle 0\rangle) \rangle \}
\]
is a prefix code strictly containing $C_x(Z)$, contradicting its
maximality.

This contradiction shows that $s_1$ is a winning strategy for Player~1,
completing the proof.
\end{proof}
\\

We now extend the guiding idea of the previous theorem to the setting of
open winning sets on a full tree over any finite alphabet.
Conceptually, both the construction and the proof follow the same pattern
as in the binary case. However, the resulting criterion requires checking a substantially larger family of
auxiliary sequences.

Let $k \geq 2$ be fixed, and denote by $M_k$ the set of all functions
$\{0,1,\dots,k-1\} \to \{0,1,\dots,k-1\}$. Given
$x=\langle x_1,x_2,\dots\rangle \in (M_k)^{\mathbb{N}}$ and a position
$p=\langle a_1,a_2,\dots,a_n\rangle \in \{0,1,\dots,k-1\}^{\leq \mathbb{N}}$,
define
\[
c_x(p)
=
\langle c_x^1(p),\dots,c_x^{\lfloor \mathrm{len}(p)/2 \rfloor}(p) \rangle,
\]
where for each $i \leq \lfloor \mathrm{len}(p)/2 \rfloor$,
\[
c_x^i(p) := x_i(a_{2i-1}) + a_{2i} \pmod{k}.
\]
As in the binary case, the role of the map $c_x(\cdot)$ is to send
positions consistent with different strategies of Player~1 to the same codewords.
\begin{theorem}
\label{theorem: max prefix equivalent finite alphabet}
Player~1 has a winning strategy in the game $(T,W)$, where
\[
T = \{0,1,\dots,k-1\}^{\leq \mathbb{N}}, \qquad
W = \bigcup_{p \in Z} [T_p],
\]
and $Z \subset T$ is a minimal-size set, if and only if for every
$x \in (M_k)^{\mathbb{N}}$, the set
\[
C_x(Z) := \{ c_x(p) : p \in Z \}
\]
is a maximal prefix code over the alphabet $\{0,1,\dots,k-1\}$.
\end{theorem}
\begin{proof}
The argument follows the same scheme as in the binary case. We therefore
focus on the nontrivial direction.

Assume that for every $x \in (M_k)^{\mathbb{N}}$ the set $C_x(Z)$ is a maximal
prefix code. We define a strategy $s_1$ for Player~1 as follows. For every
even-length position $p \in T$, let
\[
s_1(p) =
\begin{cases}
a \in \{0,1,\dots,k-1\}, & \text{if for all } i \in \{0,1,\dots,k-1\},\;
p \circ \langle a,i\rangle \text{ is a prefix} \\
& \text{of some position in } Z, \\
0, & \text{if no such } a \text{ exists}.
\end{cases}
\]
$s_1$ is a well-defined strategy.

We claim that $s_1$ is a winning strategy for Player~1. To this end, it
suffices to show that for every strategy $s_2$ of Player~2 and for every
even-length position $p$ arising along the play $x(s_1,s_2)$ (and which is not already in $W$), there exists
an action $a \in \{0,1,\dots,k-1\}$ such that
\[
\forall i \in \{0,1,\dots,k-1\}, \quad
p \circ \langle a,i\rangle
\]
is a prefix of some position in $Z$. Since every continuation remains a prefix of some element of $Z$, the resulting play remains in $W$.

Suppose, to the contrary, that there exists a strategy $s_2$ of
Player~2 and a minimal $t \in \mathbb{N}$ such that
\[
p := x(s_1,s_2)_{[2t]}
\]
fails to satisfy this property. Then for every $i \in \{0,1,\dots,k-1\}$
there exists $a_i \in \{0,1,\dots,k-1\}$ such that
$p \circ \langle i,a_i\rangle$ is not a prefix of any position in $Z$.

Define a sequence $x \in (M_k)^{\mathbb{N}}$ by
\[
x_t(i) := -a_i \pmod{k},
\qquad
x_j := \mathrm{id} \quad \text{for all } j \neq t.
\]
Now consider any position $p' \in Z$ having $p$ as a prefix. By construction,
the $t$-th coordinate of $c_x(p')$ satisfies
\[
c_x^t(p') = x_t(i) + a_i \not\equiv 0 \pmod{k}.
\]
Consequently, every element of $C_x(Z)$ extending $c_x(p)$ has a nonzero
$t$-th coordinate. It follows that
\[
C_x(Z) \cup \{\, c_x(p) \circ \langle 0 \rangle \,\}
\]
is still a prefix code, strictly containing $C_x(Z)$, contradicting the
assumption that $C_x(Z)$ is maximal.

This contradiction shows that such a position $p$ cannot exist, and hence
$s_1$ is a winning strategy for Player~1.
\end{proof}
\begin{remark}
    Note that the number of auxiliary sequences that must actually be examined is smaller than indicated in the previous theorem. For instance, it suffices to consider a single equivalence class of $M_k$, namely the one determined by the value at $0$. Substituting $k=2$ then yields Theorem~\ref{thm:binary-tree-prefix}.
\end{remark}

We now turn to the proofs of the main results. Before doing so, we briefly review free groups and Nielsen–Schreier theory, which will play a central role in the arguments that follow.

\section{Introduction to Nielsen-Schreier Theory}
\label{section:Introduction To Nielsen-Schreier Theory}
We begin with a brief survey of free groups and their graph-theoretic representations.

\begin{definition}
A \emph{free group} $F$ of rank $n$ is the group of all words in a given generating set $\mathcal{A} = \{ a_1, \dots, a_n \}$ with no relations among the generators. That is,
\[
F \equiv F(\mathcal{A}) = \langle a_1, \dots, a_n \mid \ \rangle.
\]
\end{definition}

One motivation for considering free groups in the context of games on infinite trees is that elements of a free group can be naturally identified with the positions of an infinite directed tree (see Figure~1). 

In 1921, Jacob Nielsen~\cite{nielsen21} proved that any finitely generated subgroup of a free group is itself free. Schreier~\cite{Schreier27} later extended this result, providing a formula for the minimal number of generators of a free subgroup. This extension led to the introduction of what are now known as \emph{Schreier graphs}. 

At the heart of the Nielsen-Schreier theory is a correspondence between free subgroups $H \leq F(\mathcal{A})$ and directed graphs with a distinguished vertex $o$, whose edges are labeled by elements of $\mathcal{A}$, such that each vertex has exactly one outgoing edge and exactly one incoming edge for every $a \in \mathcal{A}$.

To make this correspondence precise, we recall the definition of the Schreier coset graph.

\begin{definition}
Given a subgroup $H \leq F(\mathcal{A})$, the \emph{Schreier graph} 
\[
Sch(H \backslash F(\mathcal{A}), \mathcal{A}),
\] 
denoted simply $Sch(H)$ when the subgroup is clear from context, is the graph whose vertex set
\[
V(Sch(H)) = H \backslash F(\mathcal{A})
\] 
consists of the right cosets of $H$ in $F(\mathcal{A})$. Each vertex $Hw \in H \backslash F(\mathcal{A})$ is connected to $Hwa$ by a directed edge labeled $a$, for every $a \in \mathcal{A}$. We denote such an edge by
\[
Hw \xrightarrow{a} Hwa.
\]
\end{definition}

The Schreier graph $Sch(H)$ has a distinguished vertex corresponding to the coset $H$, which we denote by $o$. 

In Figure~1, we illustrate the Schreier graphs of two free subgroups of $F(\mathcal{A})$ with $\mathcal{A} = \{a,b\}$, one of which being the trivial subgroup, denoted $\langle \rangle$, containing only the empty word\footnote{The Schreier graph of the trivial subgroup is better known as the \emph{Cayley graph} of the free group.}.  

\medskip

The order of the Schreier graph satisfies
\[
|V(Sch(H))| = |H \backslash F(\mathcal{A})| = [F(\mathcal{A}):H],
\]
that is, it equals the index of the subgroup $H$ in $F(\mathcal{A})$. Moreover, a Schreier graph is \emph{regular}: every vertex has exactly $|\mathcal{A}|$ outgoing edges and $|\mathcal{A}|$ incoming edges (e.g. \textsection 4 in \cite{Schreier27}).

There is a natural one-to-one correspondence between words in the free group $F(\mathcal{A})$ and non-backtracking walks in $Sch(H)$~\cite{GRIGORCHUK20121340}. Under this correspondence, an element $w \in F(\mathcal{A})$ belongs to the subgroup $H \leq F(\mathcal{A})$ if and only if the associated walk in $Sch(H)$, starting at the distinguished vertex $o$, also terminates at $o$. For instance, the Schreier graph of the trivial subgroup is an infinite directed tree, so there are no nontrivial closed walks based at $o$, reflecting the fact that the subgroup contains only the empty word.

We next recall the definition of \emph{Stallings core graphs}~\cite{GRIGORCHUK20121340}. A \emph{hanging branch} in a Schreier graph is a subgraph that is isomorphic, as a directed graph, to an infinite tree whose root has degree 1. In Figure~1, for example, the Schreier graph $Sch(\langle b, aba, ab^{-1}a \rangle)$ contains two hanging branches, each rooted at a vertex connected to $o$ by an edge labeled $a$ (one incoming and one outgoing).

The \emph{core graph} of a subgroup $H \leq F(\mathcal{A})$, denoted $Core(H)$, is obtained by \say{trimming off} all hanging branches from the Schreier graph. More precisely, one removes all vertices of the hanging branches, except for their roots, along with all edges incident to these vertices.

\begin{figure}[htbp]
    \centering
    \includegraphics[width=0.9\linewidth]{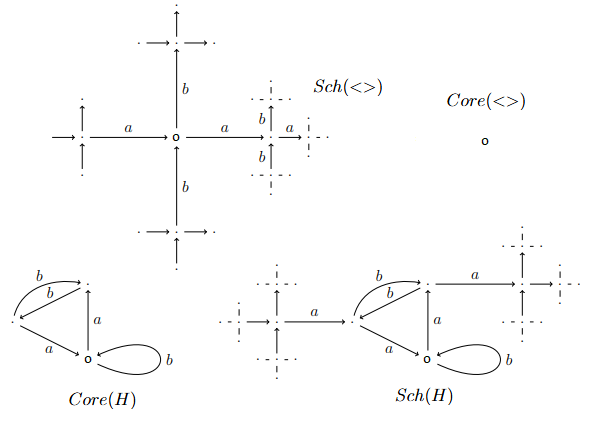}
    \caption{Illustration of the Schreier graph $\mathrm{Sch}(\langle\,\rangle)$, its core $\mathrm{Core}(\langle\,\rangle)$, the Schreier graph $\mathrm{Sch}(H)$, and its core $\mathrm{Core}(H)$ for the free group $F = F(\{a,b\})$ and the subgroup $H = \langle b,\, aba,\, ab^{-1}a \rangle$. The vertices of each Schreier graph are in bijection with the right cosets in $H \backslash F(\mathcal{A})$. In this example, the coset set decomposes as $\{H,\, Ha,\, Hab\} \;\cup\; \{Ha^{2}w,\; Ha^{-2}w' \mid w,w' \in F(\mathcal{A}) \text{ with no cancellation}\}$.}
    \label{fig:sch-graphs-example}
\end{figure}

The following proposition \cite{GRIGORCHUK20121340} characterizes finitely generated free subgroups in terms of the core graph. 
\begin{proposition} \label{proposition: core graph is finite}
A subgroup $H \leq F$ is finitely generated if and only if $Core(H)$ is a finite graph. 
\end{proposition}
We also introduce the concept of \emph{Stalling folding} \cite{Stallingfold}. Consider a directed graph, with edges labeled by elements of the alphabet $\mathcal{A}$, the process of Stallings folding includes identifying edges both labeled by the same element $a \in \mathcal{A}$ with the same endpoints and coinciding those edges. Consider a directed graph, with its edges labeled by elements of the alphabet $\mathcal{A}$. Two edges $e_1=\{v,u_1\} , \{v,u_2 \}$ labeled by the same element of $\mathcal{A}$, can be folded if they share a vertex $v$, and both exit or enter vertex $v$. Under those conditions a Stallings fold, is simply to contract the vertices $u_1,u_2$, or “gluing" the edges $e_1,e_2$ into a new edge, with the same label and direction; see Figure 2. 

\begin{figure}[]
    \centering
    \includegraphics[width=0.9\linewidth]{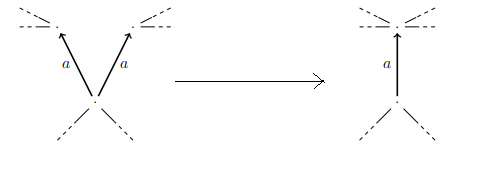}
    \caption{ A Stallings folding of two edges, labeled $a$ with the same starting vertex and direction}
\end{figure}

The primary use of Stallings folding in this work will be the fact that the core graph of a free subgroup can be obtained via a series of Stallings foldings starting at the bouquet graph, that is, a graph where each cycle corresponds to a generating word of the subgroup; see Figure 3. The process of Stallings folding is recursive and does not depend on the order in which the foldings are performed \cite{KAPOVICH2002608}.
\\
\begin{figure}[]
    \centering
    \includegraphics[width=0.9\linewidth]{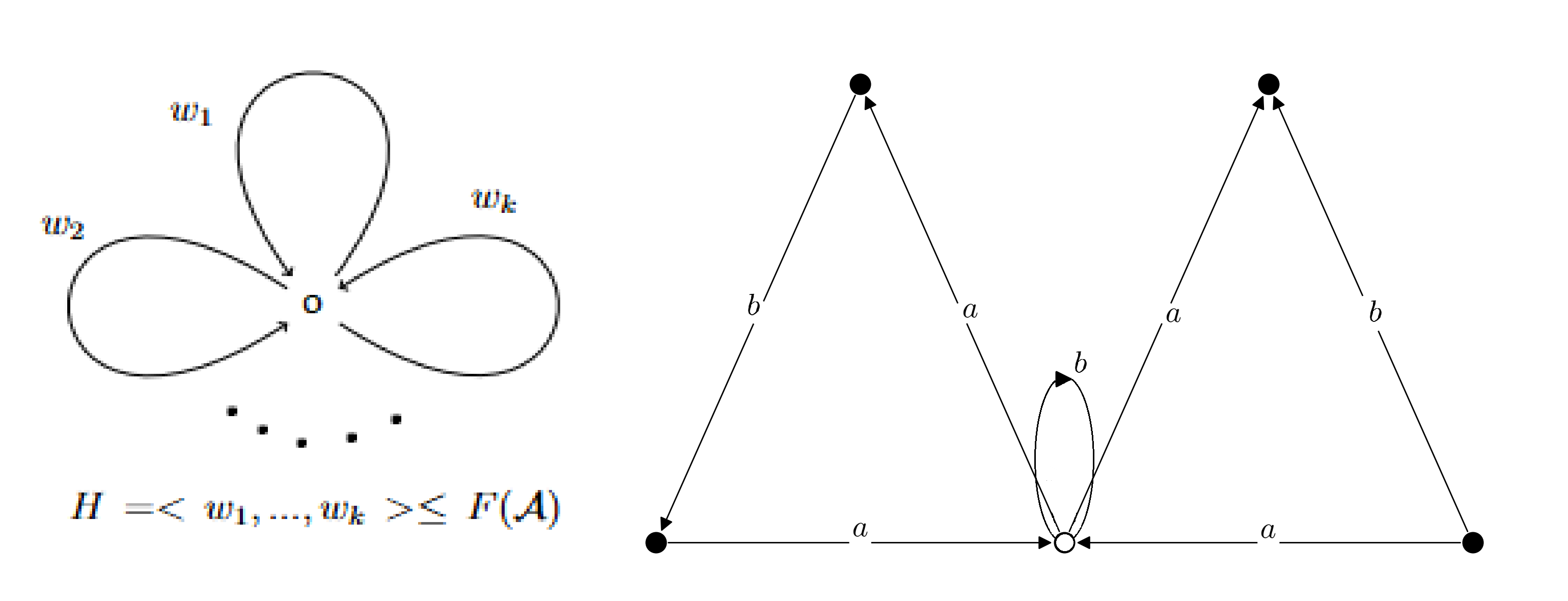}
    \caption{A diagram of a bouquet graph and the bouquet graph of the set from previous example $\{ aba, ab^{-1}a ,b\}$}
\end{figure}
\\
When all vertices of the core graph are full-degree (both entering and exiting), the core graph and the Schreier graph coincide, $Core(H)=Sch(H)$ \cite{KAPOVICH2002608}. The reader may verify that if it was not the case, then a hanging tree would have been removed, thus lowering the degree in the corresponding vertex. Alternatively, the statement could be formulated as follows.
\begin{proposition} \label{proposition: finite index mean core = sch}
    When $|\mathcal{A}|<\infty $ and $H\leq F(\mathcal{A})$ is finitely generated, $|V(Sch(H))| = [F(\mathcal{A}):H]< \infty$ if and only if the core graph is equal to the Schreier graph.
\end{proposition} 
We end the introduction to Nielsen-Schreier theory, by stating the following result, known as the Nielsen-Schreier index formula (which we will use in Theorem ~\ref{corollary: using max prefix codes}):
\begin{theorem} \label{theorem: index formula}
    if $G$ is a free group of rank $k$, and $H$ is a subgroup of finite index $[G : H] = d$, then $H$ is free of rank $d(k-1)+1$. 
\end{theorem}

\section{Algebraic Criteria Using Maximal Prefix Codes} \label{section:alg conditions}
We now return to the original goal of this paper, namely to provide algebraic
criteria for determining which player has a winning strategy. To this end, we
exploit algebraic properties of maximal prefix codes in order to derive a
strong necessary condition for Player~1 to win.

In particular, from Theorem~\ref{theorem: max prefix equivalent finite alphabet}
we infer the following necessary condition for the existence of a
winning strategy for Player~1.

\begin{theorem}
\label{corollary: using max prefix codes}
Let $T=\{0,1,\dots,k-1\}^{\leq \mathbb{N}}$, let $Z \subseteq T$, and let
\[
W=\bigcup_{p\in Z}[T_p]
\]
be a minimal-size open winning set. If there exists
$x\in (M_k)^{\mathbb{N}}$ such that for every subset
$Z' \subseteq Z$ of cardinality
\[
\left(\max_{p\in Z} \left\lfloor \frac{\mathrm{len}(p)}{2} \right\rfloor \right)
\cdot (|\mathcal{A}|-1) + 1
\]
the subgroup generated by $C_x(Z')$, that is $H=\langle C_x(Z') \rangle$, has infinite index in $F(\mathcal{A})$,
\[
[F(\mathcal{A}) : H] = \infty,
\]
then Player~2 has a winning strategy in the game $(T,W)$.
\end{theorem}
\begin{proof}
Suppose, that Player~1 has a winning strategy.
Then, by Theorem~\ref{theorem: max prefix equivalent finite alphabet}, for every
$x \in (M_k)^{\mathbb{N}}$, the set $C_x(Z)$ is a maximal prefix code. In
particular, the subgroup $\langle C_x(Z) \rangle$ has finite index in
$F(\mathcal{A})$.

Moreover, every maximal prefix code contains a subset of at most 
\[
\big(\max_{c \in C_x(Z)} \mathrm{len}(c)\big)(|\mathcal{A}|-1) + 1
\]
elements that generate the group. To see this, let $c_x(p_1) = \langle a^1_1,\dots,a^1_n \rangle$
be a word of maximal length in $C_x(Z)$. By maximality, for every $a \in \mathcal{A}$,
the word $\langle a^1_1, \dots, a^1_{n-1}, a \rangle$ must belong to $C_x(Z)$
(otherwise one could extend the code). Iteratively, for each $a \in \mathcal{A}$,
there exists $p_2 \in Z$ such that $\langle a^1_1, \dots, a^1_{n-2}, a \rangle$
is a prefix of some $c_x(p_2)$. Continuing this construction yields a
subset $Z'$ of the desired size. We claim that $C_x(Z')$ is a generating set of a finite index free subgroup, thus finishing the proof.

\medskip

Denote $H=\langle C_x(Z') \rangle$.
We show that every vertex in $Core(H)$ has full degree, i.e.,  
\[
deg_{\mathrm{out}}(v) = deg_{\mathrm{in}}(v) = |\mathcal{A}|
\qquad \text{for all } v \in Core(H),
\]
which implies that $Core(H)=Sch(H)$. Since $H$ is finitely generated, this forces the Schreier graph to be finite (see Proposition~\ref{proposition: finite index mean core = sch}), and therefore $[F(\mathcal{A}) : H] = |Sch(H)| < \infty$, as required.

\medskip 

Let $v $ be any vertex of $Core(H)$. Such vertices arise as points lying on the
“generator cycles” in the bouquet before folding. After folding firstly the elements 
\[
\langle a^1_1, \dots, a^1_{n-1}, a \rangle , \ \ a \in \mathcal{A}
\]
We notice that the vertex corresponding to \[
\langle a^1_1, \dots, a^1_{n-1} \rangle
\] 
has full exiting degree. After folding the additional elements which have 
\[
\langle a^1_1, \dots, a^1_{n-2}, a \rangle , \ \ a \in \mathcal{A}
\]
as prefixes, we observe that the number of vertices in the graph which do not have full exiting degree is reduced by (at least) one. Thus, after iteratively we fold all the remaining elements in $C_z(Z')$ we get a graph with full \emph{exiting} degree, that is, we obtain $deg_{\mathrm{out}}(v) = |\mathcal{A}|$ for all $v\in Core(H)$.

\smallskip

To see that every vertex also has full \emph{entering} degree, note first that $Core(H)$ is finite (Proposition~~\ref{proposition: core graph is finite}).  
Since the total incoming degree equals the total outgoing degree,
\[
\sum_{v} deg_{\mathrm{in}}(v)
  = \sum_{v} deg_{\mathrm{out}}(v)
  = |Core(H)| \cdot |\mathcal{A}|,
\]
and since each $deg_{\mathrm{in}}(v) \le |\mathcal{A}|$, it follows that equality must hold at every vertex.  
Thus every vertex has full degree on both sides, and therefore $Core(H) = Sch(H)$.  
This completes the proof.
\end{proof}

\begin{remark}
Notice that the proof of the previous Theorem, showed that in fact, we have at least $|\{p \in Z : \lfloor \mathrm{len}(p)/2 \rfloor \text{ is maximal} \}|/|\mathcal{A}|\geq 1$ different ways to choose $Z'$, as the initial choice of the maximal length word was arbitrary.    
\end{remark}

The following Corollary also follows immediately from the classical result that every maximal prefix code generates a free subgroup of finite index.

\begin{corollary}
    Under the same conditions as in the previous Theorem, if there exists $x \in (M_k)^\N$ such that  $[F(\mathcal{A}):\langle C_x(Z) \rangle]= \infty$, then Player~2 has a winning strategy.
\end{corollary}
Also as a Corollary we get also Theorem~\ref{Main result: index property}.

\medskip

\begin{proof} [Proof of Theorem \ref{Main result: index property}]
    Let $T=\{0,1,\dots,k-1\}^{\leq \mathbb{N}}$, let $Z \subseteq T$
    Notice that for the constant vector $0^\N \in (M_k)^\N$, which is such that $0^\N_i(j)=0$ for every $i \in \N , j\in \{0,\dots,k-1\}$, we have $C_{0^\N}(Z)=\hat{Z}$
\end{proof}
\begin{remark}    
Observe that the index of $\langle C_x(Z) \rangle$ (which is also the size of the Schreier graph) is bounded by the maximal
length of elements in $C_x(Z)$:
\[
\max_{p \in Z} \mathrm{len}(c_x(p)) = \max_{p \in Z} \left\lfloor \frac{\mathrm{len}(p)}{2} \right\rfloor.
\]

By the Nielsen-Schreier formula (Theorem~\ref{theorem: index formula}), the
rank of a finite-index subgroup $H \leq F(\mathcal{A})$ satisfies
\[
\mathrm{rank}(H) = [F(\mathcal{A}):H] \cdot (|\mathcal{A}|-1) + 1.
\]
Applied to $H = \langle C_x(Z) \rangle$, it follows that
\[
rank(\langle C_x(Z) \rangle) \le [F(\mathcal{A}) : \langle C_x(Z) \rangle] \cdot (|\mathcal{A}|-1) + 1
\leq \max_{p \in Z} \left( \left\lfloor \frac{\mathrm{len}(p)}{2} \right\rfloor \right) \cdot (|\mathcal{A}|-1) + 1.
\]
Moreover, the size of the subset $Z'$ in the Theorem is tight (as can be seen in the case where all the positions in $Z$ have the same length).
\end{remark}

\medskip

Two natural questions arise at this point. First, can our results concerning the finite index condition be extended to the case of games on non-full trees? Second, is this condition optimal, in the sense that it cannot be strengthened in general? The finite index property \cite{FiniteIndexProperty} provides an affirmative answer to both questions. We formulate this answer in the form of the following lemma, using the same notation and terminology as in \cite{FiniteIndexProperty}. We begin by recalling the necessary definitions.

\begin{definition}
\begin{enumerate}
    \item A \emph{bifix code} is a prefix code with the additional property that no word in the code is a suffix of another.
    
    \item Let $\mathcal{A}$ be a finite alphabet and let $S \subset \mathcal{A}^{\le \mathbb{N}}$. We say that $S$ is \emph{uniformly recurrent} if the following two conditions hold:
    \begin{enumerate}
        \item (\emph{Closure under subwords})  
        Whenever a word $w$ belongs to $S$, every subword of $w$ also belongs to $S$.
        
        \item (\emph{Uniform recurrence})  
        For every word $w \in S$, there exists an integer $N \geq 1$ such that every word $w' \in S$ with
        $len( w')  \geq N$ contains $w$ as a subword.
    \end{enumerate}
    
    \item For a word $w \in S$, define
    \[
    L(w) = \{ a \in \mathcal{A} \mid aw \in S \}, \qquad
    R(w) = \{ b \in \mathcal{A} \mid wb \in S \}.
    \]
    We associate to $w$ an undirected bipartite graph $G(w)$ with vertex set $L(w) \sqcup R(w)$, where an edge connects $a \in L(w)$ to $b \in R(w)$ if and only if $awb \in S$. If $G(w)$ is a tree for every word $w \in S$, we say that $S$ is a \emph{tree set}.
\end{enumerate}
\end{definition}

Recall that for a position $p=\langle a_1,\dots,a_n\rangle$ we denote
\[
\hat{p} = \langle a_2,\dots,a_{\lfloor n/2 \rfloor} \rangle .
\]

\begin{lemma}
Let $S$ be a uniformly recurrent tree set containing $\mathcal{A}$, and let \footnote{With respect to any fixed strategy $s_1\in S_1(\mathcal{A}^{\le \N})$.} $T = Z(S) \subset \mathcal{A}^{\le \mathbb{N}}$. Let $Z \subset T$ be such that $\hat{Z}$ is a bifix code, and let
\[
W = \bigcup_{p \in Z} [T_p]
\]
be the corresponding open subset of $T$. Then Player~1 has a winning strategy in the game $(T,W)$ if and only if $\hat{Z}$ is a basis of a finite index subgroup of the free group $F(\mathcal{A})$. In particular,
\[
[F(\mathcal{A}) : \langle \hat{Z} \rangle] < \infty .
\]
\end{lemma}

\begin{proof}
Player~1 has a winning strategy if and only if $\hat{Z}$ is an $S$-maximal bifix code. By Theorem~4.4 of \cite{FiniteIndexProperty}, this holds if and only if $\hat{Z}$ is a basis of a finite index subgroup of $F(\mathcal{A})$.
\end{proof}

\section{Covering the Tree by a Schreier Graph} \label{section:coverings}

In this section, we show that several results concerning maximal prefix codes can be obtained by applying Martin’s notion of covering to the game tree, using the infinite tree associated with the free group.
First, we introduce the concept of coverings, as presented in~\cite{Solan} and ~\cite{MartinCovering}.
\begin{definition}
    Let $T$ be a tree. A \emph{covering} is a quadruple $(\tilde{T} , \pi, \varphi_1
, \varphi_2)$, where $\tilde{T}$ is a tree, $\pi : \tilde{T} \to T, \varphi_{1,2}
: S_{1,2}(\tilde{T}) \to S_{1,2}(T)$, such that the following properties hold:

$\bullet$ The function $\pi$ sends positions of length $n$ to positions of length $n$:
 $len(\tilde{p}) = len(\pi(\tilde{p})), \forall \tilde{p} \in \tilde{T}$.
In addition, under $\pi$, two sequences that extend one another are mapped
to sequences that extend one another.
$\tilde{p},\tilde{p} \in \tilde{T} , \tilde{p} \preceq \tilde{p'} \Rightarrow \pi(\tilde{p}) \preceq \pi(\tilde{p'})$.
This in particular implies that $\pi$ defines a function from $[\tilde{T}]$ to $[T]$.
This map is also denoted by $\pi$.

$\bullet$ For each $i = 1,2$, the move $\varphi_i(\tilde{s_i})(p)$ depends on the behavior of $\tilde{s_i}$ for
positions of length at most $len(p)$. That is, for every $\tilde{s_i} , \tilde{s'_i}$ and
every $n \in \N$, if $\tilde{s_i}(\tilde{p}) = \tilde{s'_i}(\tilde{p})$ for every position $\tilde{p} \in \tilde{T}$ of length at most
$n$, then $\varphi_i(\tilde{s_i})(p) = \varphi_i(\tilde{s'_i}
)(p)$ for every position $p \in T$ of length at most
$n$.

$\bullet$ For each $i = 1,2$ and every strategy $\tilde{s_i} \in S_i(\tilde{T})$, if the play $x$ is consistent with $\varphi_i(\tilde{s_i})$, then there is a play $\tilde{x}$ that is consistent with $\tilde{s_i}$ and
satisfies $\pi(\tilde{x}) = x$.
\end{definition}
The concept of covering is useful because of the following result, due to Martin~\cite{MartinCovering}.
\begin{lemma} \label{lemma: martin covering lemma}
    Let $(T, W)$ be a game, let $(T , \pi, \varphi_1 , \varphi_2 )$ be a covering
of $T$, and let $i \in {1,2}$. If Player $i$ has a winning strategy $\tilde{s_i}$ in the game $(\tilde{T} , \pi^{-1}(W))$, then $\varphi_i(\tilde{s_i})$ is a winning strategy of Player $i$ in $(T,W)$. In
particular, if the games are determined, then the same player has a winning strategy in both $(\tilde{T},\pi^{-1}(W))$ and $(T,W)$.

\end{lemma}

We now introduce a covering of the tree 
\[
T = \mathcal A^{<\mathbb N}
\] 
using the Schreier graph of the trivial subgroup of the free group $F(\mathcal A)$, as discussed in the previous section.

\begin{lemma} \label{lemma: covering by schreier graph}
Let 
\[
W = \bigcup_{p \in Z} [T_p],
\] 
and denote the positions in $Z$ by 
\[
p = \langle a_1,\dots,a_{n} \rangle \in Z.
\]  
Then the tree $Sch(\langle \rangle)$ provides a covering of $(T,W)$.  
In this covering, 
\[
\pi^{-1}(W) = W' = \bigcup_{p = \langle a_1,...,a_n \rangle \in Z} \bigcup_{\epsilon_1,\dots,\epsilon_{n} \in \{\pm 1\}} [T_{\langle a_1^{\epsilon_1},\dots,a_{n}^{\epsilon_{n}} \rangle}],
\]
and the union is taken such that no reduction (that is, segments of the form $aa^{-1}$ or $a^{-1}a$) appears in $a_1^{\epsilon_1}\dots a_{n}^{\epsilon_{n}}$.
\end{lemma}

\begin{remark}
    From now on, whenever we write $\langle a_1^{\epsilon_1},\dots,a_{n}^{\epsilon_{n}} \rangle$ we assume that no reduction occurs.
\end{remark}
\begin{proof}
We construct a covering of $T$ as follows. Define 
\[
\pi : Sch(\langle \rangle) \to T, \quad \pi(\langle a_1^{\epsilon_1},\dots,a_{n}^{\epsilon_{n}}\rangle) = \langle a_1,\dots,a_{n}\rangle
\] 
Where $a_i \in \mathcal{A}, \epsilon_i \in \{\pm 1\}, \quad 1\le i \le n$.

Similarly, define 
\[
\varphi_j : S_j(Sch(\langle \rangle)) \to S_j(T), \quad j=1,2,
\] 
by setting $\varphi_j(s_j)(p) = a$ whenever $s_j(p) = a$ or $a^{-1}$.

One verifies all covering properties by direct inspection. For instance, if $x \in [T]$ corresponds to strategies $\varphi_j(s_j)$, i.e., 
\[
x = x(\varphi_1(s_1), \varphi_2(s_2)),
\] 
then let 
\[
x' = x'(s_1, s_2) \in [Sch(\langle \rangle)].
\] 
Notice that for every $n$, 
\[
x'[n] = x[n] \quad \text{or} \quad x[n]^{-1},
\] 
so $\pi(x') = x$.  

Hence, $(Sch(\langle \rangle), \pi, \varphi_1, \varphi_2)$ indeed forms a covering.   
Finally, observe that 
\[
\pi^{-1}(W) = W',
\] 
which completes the proof.
\end{proof}

\medskip

As a direct consequence of Lemma~\ref{lemma: covering by schreier graph} and Lemma~\ref{lemma: martin covering lemma}, we obtain:

\begin{lemma} \label{lemma:cover by sch}
Let 
\[
W = \bigcup_{p = \langle a_1, \dots, a_{n} \rangle \in Z} [T_p].
\]  
Then Player~1 can win the game $(T,W)$ if and only if she can win the game $(Sch(\langle \rangle), W')$, where
\[
W' = \bigcup_{p = \langle a_1, \dots, a_{n} \rangle \in Z} \bigcup_{\epsilon_1, \dots, \epsilon_{n} \in \{\pm 1\}} [T_{\langle a_1^{\epsilon_1}, \dots, a_{n}^{\epsilon_{n}} \rangle}].
\]
\end{lemma}

We now turn to the proof of Theorem~\ref{main result: max prefix code ineq}. We begin by establishing the following Proposition.

\begin{proposition} \label{Theorem: Kraft like inequality with coverings}
Let $C$ be a maximal prefix code, and let $x \in \mathcal{A}^{\N}$. Then
\[
\sum_{c = \langle c_1, \ldots, c_n \rangle \in C}
2^{|\{ i \in [n] : \, c_i \neq x_i \}|}
\cdot \frac{1}{(2|\mathcal{A}| - 1)^{\mathrm{len}(c)}} \geq 1.
\]
\end{proposition}
\begin{proof} 
Let $C$ be a maximal prefix code, and let 
\[
x=\langle x_1,x_2,\ldots\rangle \in \mathcal{A}^{\mathbb{N}}
\]
(if $C$ is finite, we may instead take $x\in \mathcal{A}^{\max_{c\in C}\mathrm{len}(c)}$). Using Lemma~\ref{lemma:create winning set from max prefix code}, we construct an open winning set whose unique winning strategy for Player~1 is prescribed by $x$. Namely, define
\[
Z(C)=\bigl\{\, p=\langle x_1,c_1,x_2,c_2,\ldots,x_n,c_n\rangle : \langle c_1,\ldots,c_n\rangle\in C \,\bigr\},
\qquad 
W=\bigcup_{p\in Z(C)} [T_p].
\]

We now consider the lift $\pi^{-1}(W)$ of $W$ to the tree $\mathrm{Sch}(\langle\rangle)$. By Lemma~\ref{lemma:cover by sch}, Player~1 has a winning strategy in the game $\bigl(\mathrm{Sch}(\langle\rangle),\pi^{-1}(W)\bigr)$. Fix such a strategy, and consider the set of positions that are consistent with it. Our goal is to count how many of these positions are mapped, under the covering map $\pi$, to a single position in $T$.

Under $\pi$, all such positions are mapped to positions in the tree $\mathcal{A}^{\le \N}$ of the form
\[
p=\langle x_1,c_1,x_2,c_2,\ldots,x_n,c_n\rangle,
\qquad
\text{with }\langle c_1,\ldots,c_n\rangle\in C.
\]
For each $1\le i\le n$, we count the number of possible configurations for the corresponding $(2i-1,2i)$-entries of a lifted position.

\medskip

Suppose first that $x_i\neq c_i$. In this case, the position $p$ lifts to distinct positions in $\pi^{-1}(W)$, whose $(2i-1,2i)$-entries realize both possibilities
\[
\langle \tilde{x}_i, c_i\rangle
\quad\text{and}\quad
\langle \tilde{x}_i, c_i^{-1}\rangle,
\]
where $\tilde{x}_i$ denotes the corresponding lift of $x_i$ in $\mathrm{Sch}(\langle\rangle)$.

If, on the other hand, $x_i=c_i$, then $p$ lifts to positions whose $(2i-1,2i)$-entry contains only one of
\[
\langle \tilde{x}_i, c_i\rangle
\quad\text{or}\quad
\langle \tilde{x}_i, c_i^{-1}\rangle,
\]
depending on whether $\tilde{x}_i=x_i$ or $\tilde{x}_i=x_i^{-1}$ (see Figure~4 for the first case and Figure~5 for the latter).

\medskip 

It follows that the number of positions in $\mathrm{Sch}(\langle \rangle)$ consistent with a unique winning strategy of Player~1 that are mapped to a given position
\[
p = \langle x_1, c_1, \ldots,x_n, c_{n} \rangle \in T
\]
is exactly
\[
2^{|\{ i \in [n] : \, c_i \neq x_i \}|}.
\]
Identifying $\mathrm{Sch}(\langle \rangle)$ with a full tree whose alphabet has size $2|\mathcal{A}|-1$, and combining Lemma~\ref{lemma: necceserry open condition lengths} with the observation above, we obtain the desired result.
\end{proof}

\begin{figure}[htbp]
    \centering
    \includegraphics[width=0.5\linewidth]{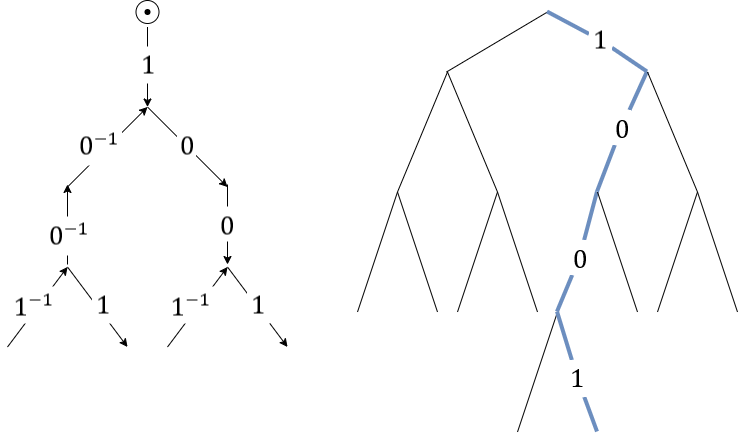}
    \caption{For example, let $x = \langle 1,0 \rangle$ and $c = \langle 0,1 \rangle$, which together determine the position $\langle 1,0,0,1 \rangle$ in the full binary tree $\{0,1\}^{\leq \mathbb{N}}$. See the diagram on the right, where this position is highlighted. The corresponding positions in the covering Schreier graph associated with $x = \langle 1,0 \rangle$ and $c = \langle 0,1 \rangle$ are $\langle 1,0,0,1 \rangle$, $\langle 1,0,0,1^{-1} \rangle$, $\langle 1,0^{-1},0^{-1},1 \rangle$, $\langle 1,0^{-1},0^{-1},1^{-1} \rangle$. See the diagram on the left.}
    \label{figure-covering-of-1001}
\end{figure}
\begin{figure}[htbp]
    \centering
    \includegraphics[width=0.5\linewidth]{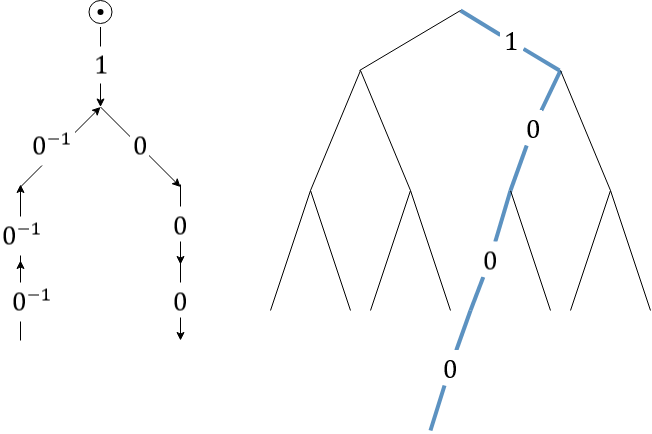}
    \caption{For example, let $x = \langle 1,0 \rangle$ and $c = \langle 0,0 \rangle$, which together determine the position $\langle 1,0,0,0 \rangle$ in the full binary tree $\{0,1\}^{\leq \mathbb{N}}$. See the diagram on the right, where this position is highlighted. The corresponding positions in the covering Schreier graph associated with $x = \langle 1,0 \rangle$ and $c = \langle 0,0 \rangle$ are $\langle 1,0,0,0 \rangle$, $\langle 1,0^{-1},0^{-1},0^{-1} \rangle$. See the diagram on the left.}
    \label{covering-1000}
\end{figure}

\newpage

The following examples demonstrate that these bounds are sharp.

\begin{example}
\begin{enumerate}
\item Let
\[
C = \{ \langle 1 \rangle, \langle 0,1 \rangle, \langle 0,0,1 \rangle, \langle 0,0,0 \rangle \},
\quad x = \langle 1,1,1,\ldots \rangle.
\]
Then
\[
\sum_{c \in C}
2^{|\{ i \in [n] : c_i \neq x_i \}|}
\cdot \frac{1}{(2|\mathcal{A}| - 1)^{\mathrm{len}(c)}}
=
\frac{2}{3} + \frac{2}{9} + \frac{2}{27} + \frac{1}{27}
= 1.
\]

\item Let $C = \{0,1\}^n$ and let $x \in \{0,1\}^{\mathbb{N}}$. Then
\[
\sum_{c \in C}
2^{|\{ i \in [n] : c_i \neq x_i \}|}
\cdot \frac{1}{(2|\mathcal{A}| - 1)^{\mathrm{len}(c)}}
=
\frac{1}{3^n}
\sum_{k=0}^n \binom{n}{k} 2^k
=
\frac{(1+2)^n}{3^n}
= 1.
\]
Removing any element from $C$ yields a strict inequality, corresponding to the fact that the resulting code is no longer maximal.
\end{enumerate}
\end{example}

\begin{remark}
    Let $C$ be a maximal prefix code, and let $m \ge \max_{c \in C} len(c)$. We can sum the inequality in Lemma, and get 
    \[
    |\mathcal{A}|^{n} \le \sum_{x \in \mathcal{A}^n} \sum_{c = \langle c_1, \ldots, c_n \rangle \in C}
2^{|\{ i \in [n] : \, c_i \neq x_i \}|}
\cdot \frac{1}{(2|\mathcal{A}| - 1)^{\mathrm{len}(c)}} = \sum_{c = \langle c_1, \ldots, c_n \rangle \in C} \frac{1}{(2|\mathcal{A}| - 1)^{\mathrm{len}(c)}} \sum_{x \in \mathcal{A}^n}
2^{|\{ i \in [n] : \, c_i \neq x_i \}|} =\]
\[
=\sum_{c \in C} \frac{1}{(2|\mathcal{A}| - 1)^{\mathrm{len}(c)}} |\mathcal{A}|^{n-len(c)} \sum_{k=0}^n
\binom{len(c)}{k} 2^{k} (|\mathcal{A}|-1)^k = \sum_{c \in C} \frac{1}{(2|\mathcal{A}| - 1)^{\mathrm{len}(c)}} |\mathcal{A}|^{n-len(c)} (2|\mathcal{A}| - 1)^{len(c)} =\]
    \[
    \sum_{c \in C}  |\mathcal{A}|^{n-len(c)}.
    \]
    Thus we get $1\le \sum_{c \in C}  |\mathcal{A}|^{-len(c)}=1$, meaning we could replace \say{$\le$} with \say{$=$} in the inequality.
\end{remark}
As a corollary, we obtain the main result.

\medskip

\begin{proof}[Proof of Theorem~\ref{main result: max prefix code ineq}]
Combining Proposition~\ref{Theorem: Kraft like inequality with coverings} with the preceding remark yields the first assertion of Theorem~\ref{main result: max prefix code ineq}. The second assertion then follows immediately from the first.
\end{proof}

\section{Further Discussion}
We show that several of the preceding results extend without difficulty to the case of a countable alphabet.

\medskip

Let $\mathcal{A}$ be a countable set (finite or infinite), and let $\mu$ be a probability measure on the measurable space $(\mathcal{A}, \mathcal{P}(\mathcal{A})).$ Thus $\mu$ is uniquely determined by the weights
\[
\mu_a := \mu(\{a\}), \qquad a \in \mathcal{A},
\]
which satisfy
\[
\sum_{a \in \mathcal{A}} \mu_a = 1.
\]

\begin{remark}
The assumption that the alphabet $\mathcal{A}$ is countable is essential for our arguments, as it allows the use of countable subadditivity (in particular, the union bound). It also ensures the existence of purely atomic probability measures, i.e.\ measures for which every singleton $\{a\}$ has positive mass $\mu_a > 0$.

Indeed, purely atomic probability measures can only be supported on countable sets. To see this, for each $n \in \mathbb{N}$ define
\[
A_n := \{ a \in \mathcal{A} : \mu_a > 1/n \}.
\]
Each set $A_n$ is finite, and since
\[
\mathcal{A} = \bigcup_{n \in \mathbb{N}} A_n,
\]
it follows that $\mathcal{A}$ is countable.
\end{remark}

Throughout this section, we consider the game $(T,W)$, where $T = \mathcal{A}^{\le \mathbb{N}}$ and the winning set $W$ is open and of the form
\[
W = \bigcup_{p \in Z} [T_p],
\qquad
Z \subseteq T.
\]
We begin by extending Theorem~\ref{Main result: sum condition} to this more general setting.

\begin{lemma}\label{countable main result - sums measure}
Let $\{\mu^i\}_{i=1}^\infty$ be a sequence of probability measures on $\mathcal{A}$. If
\[
\sum_{p = \langle a_1, a_2, \dots, a_n \rangle \in Z}
\;
\prod_{i=1}^{\lfloor n/2 \rfloor} \mu^i_{a_{2i}}
< 1,
\]
then Player~$2$ has a winning strategy in the game $(T,W)$.
\end{lemma}
\begin{proof}
Assume, towards a contradiction, that Player~$1$ has a winning strategy, and fix such a strategy $s_1^\ast$. Suppose that Player~$2$ plays randomly as follows: at each even stage $2i$ ($i \ge 1$), Player~$2$ independently selects an action $a \in \mathcal{A}$ according to the probability measure $\mu^i$. This induces a probability distribution on the set of Player~$2$ strategies, which we denote by $s_2$.

Let $x(s_1^\ast, s_2)$ denote the (random) play generated by the pair of strategies $(s_1^\ast, s_2)$. Since $s_1^\ast$ is assumed to be winning for Player~$1$, we have
\[
\mathbb{P}\bigl( x(s_1^\ast, s_2) \in W \bigr) = 1.
\]
We recall that $W = \bigcup_{p \in Z} [T_p]$, and
it follows that
\[
1
= \mathbb{P}\!\left( \bigcup_{p \in Z} \{ p \preceq x(s_1^\ast, s_2) \} \right)
\le
\sum_{p \in Z} \mathbb{P}\bigl( p \preceq x(s_1^\ast, s_2) \bigr),
\]
where the inequality follows from countable subadditivity.

Now fix $p = \langle a_1, a_2, \dots, a_n \rangle \in Z$. If $p$ is not consistent with the strategy $s_1^\ast$, then $\mathbb{P}\bigl( p \preceq x(s_1^\ast, s_2) \bigr) = 0$. Otherwise, the event $\{ p \preceq x(s_1^\ast, s_2) \}$ occurs precisely when Player~$2$ selects the actions $a_{2}, a_{4}, \dots, a_{2\lfloor n/2 \rfloor}$ at the corresponding stages. By independence of Player~$2$’s random choices, we obtain
\[
\mathbb{P}\bigl( p \preceq x(s_1^\ast, s_2) \bigr)
=
\prod_{i=1}^{\lfloor n/2 \rfloor} \mu^i_{a_{2i}}.
\]

Combining the above yields
\[
1
\le
\sum_{p = \langle a_1, a_2, \dots, a_n \rangle \in Z}
\prod_{i=1}^{\lfloor n/2 \rfloor} \mu^i_{a_{2i}},
\]
which contradicts the assumption of the lemma. The conclusion follows.
\end{proof}

\begin{remark}
Let $W \subseteq [\mathcal{A}^{\le \mathbb{N}}]$ be a winning set for which Player~$1$ has a unique winning strategy. Then the sum appearing in Lemma~\ref{countable main result - sums measure} is equal to $1$. Indeed, in this case the events corresponding to distinct prefixes $p \in Z$ are mutually disjoint, since the probability that a given position is a prefix of two distinct plays is zero. Consequently, the inequality arising from the union bound becomes an equality.
\end{remark}

\medskip

By suitably modifying Definition~\ref{def of prefix codes} to allow a countable alphabet, one readily verifies that Lemma~\ref{lemma:create winning set from max prefix code} and (a slightly modified) Theorem~\ref{theorem: max prefix equivalent finite alphabet} continue to hold in this setting.

\medskip

We may likewise formulate and prove an extension of Theorem~\ref{Theorem: Kraft like inequality with coverings} to the countable-alphabet setting.

\begin{proposition}
Let $C$ be a maximal prefix code over the alphabet $\mathcal{A}$, let
$x = (x_i)_{i \in \mathbb{N}} \in \mathcal{A}^{\mathbb{N}}$, and let $\mu$ be a probability measure on $\mathcal{A}$. Then
\[
\sum_{c = \langle c_1, \ldots, c_n \rangle \in C}
2^{\lvert \{ i \in [n] : c_i \neq x_i \} \rvert}
\prod_{i=1}^{n} \frac{\mu_{c_i}}{2 - \mu_{x_i}}
= 1.
\]
\end{proposition}
\begin{proof}
We follow the proof of Proposition~\ref{Theorem: Kraft like inequality with coverings}. Consider a subtree of the covering Schreier graph $\mathrm{Sch}(\langle\rangle)$ corresponding to a winning strategy of Player~1. 

For each integer $i \ge 1$, the move $a_{2i-1}$ of Player~$1$ at stage $2i-1$ is uniquely determined along this subtree and is equal to either $x_i$ or $x_i^{-1}$. We define a family of probability measures $\{\mu^i\}_{i \ge 1}$ on $\mathcal{A} \cup \mathcal{A}^{-1} \setminus \{a_{2i-1}^{-1}\}$ as follows. For every $a \in \mathcal{A} \setminus \{x_i\}$, set
\[
\mu^i_a = \mu^i_{a^{-1}}
:= \frac{\mu_a}{2} \cdot \frac{1}{1 - (\mu_{x_i}/2)}
= \frac{\mu_a}{2 - \mu_{x_i}}.
\]
Also set $\mu^i_{a_{2i-1}} = \frac{\mu_{x_i}}{2 - \mu_{x_i}}$. Each $\mu^i$ is well defined and satisfies
\[
\sum_{a \in \mathcal{A} \cup \mathcal{A}^{-1} \setminus \{a_{2i-1}^{-1}\}} \mu^i_a =  \frac{\mu_{x_i}}{2 - \mu_{x_i}}+ 2\sum_{a \in \mathcal{A} \setminus \{x_i\}}   \frac{\mu_a}{2 - \mu_{x_i}} =\frac{\mu_{x_i}}{2 - \mu_{x_i}}+ 2\cdot \frac{1-\mu_{x_i}}{2 - \mu_{x_i}}= 1.
\]

Applying the argument used in the proof of Theorem~\ref{Theorem: Kraft like inequality with coverings}, together with Lemma~\ref{countable main result - sums measure}, to the measures $\{\mu^i\}_{i \ge 1}$ and the subtree determined by the fixed strategy of Player~$1$, yields the desired conclusion.
\end{proof}
\begin{remark}
For simplicity, the probability measures assigned at each \say{level} of the Schreier graph were chosen to be symmetric, assigning equal mass to the singleton sets $\{a\}$ and $\{a^{-1}\}$. This symmetry assumption is not necessary and the measures could be modified freely. However, doing so would substantially complicate the resulting expressions.
\end{remark}

\begin{example}
Let $\mathcal{A} = \mathbb{N}_{>0} = \{1,2,\dots\}$, and let $\mu$ be a probability measure on $\mathcal{A}$ defined by
\[
\mu(\{n\}) = a_n, \qquad n \in \mathbb{N}_{>0},
\]
where $(a_n)_{n \ge 1}$ is a sequence of non-negative real numbers satisfying
\[
\sum_{n=1}^{\infty} a_n = 1.
\]
We illustrate the results with the choice $a_n = 2^{-n}$. For every maximal prefix code $C \subseteq \mathcal{A}^{<\mathbb{N}}$,
\[
\sum_{c = \langle c_1, \ldots, c_n \rangle \in C}
2^{-\sum_{i=1}^{n} c_i}
= 1.
\]
Moreover, for every sequence $x = (x_i)_{i \in \mathbb{N}} \in \mathcal{A}^{\mathbb{N}}$,
\[
\sum_{c = \langle c_1, \ldots, c_n \rangle \in C}
2^{\sum_{i=1}^{n} (x_i - c_i + 1)\,
\mathbf{1}_{\{x_i \neq c_i\}}}
\prod_{i=1}^{n} \frac{1}{2^{x_i+1} - 1}
= 1.
\]
\end{example}

\begin{example}
Let $\mathcal{A}$ be a finite alphabet, and define the uniform probability measure
\[
\mu_a = \frac{1}{|\mathcal{A}|}, \qquad a \in \mathcal{A}.
\]
In this case, the results of the present section reduce to the main results established in Section~\ref{section: perliminaries open sets}.
\end{example}

\end{document}